\documentclass[12pt, ams fonts]{amsart}

\usepackage{amsmath}
\usepackage{amsmath,amstext,amssymb,amsthm,amsfonts}
\usepackage{xcolor}

\usepackage[all]{xy}
\usepackage{amscd}

\newtheorem{theorem}{Theorem}[section]
\newtheorem{lemma}[theorem]{Lemma}
\newtheorem{corollary}[theorem]{Corollary}
\newtheorem{proposition}[theorem]{Proposition}
\newtheorem{example}[theorem]{Example}
\newtheorem{remark}[theorem]{Remark}

\newcommand{\tr}{\operatorname{tr}}

\input xy
\xyoption{all}

\begin{document}
\title{Mixed tensor products and Capelli-type determinants}

\author{Dimitar Grantcharov} 
\address{Department of Mathematics \\ University of Texas at Arlington \\ Arlington, TX 76019} \email{grandim@uta.edu}

\author{Luke Robitaille} 

\begin{abstract}
In this paper we study  properties of a homomorphism $\rho$ from the universal enveloping algebra $U=U(\mathfrak{gl}(n+1))$ to a tensor product of an algebra $\mathcal D'(n)$ of differential operators and $U(\mathfrak{gl}(n))$. We  find a formula for the image of the Capelli determinant of $\mathfrak{gl}(n+1)$ under $\rho$, and, in particular, of the images under $\rho$ of the Gelfand generators of the center $Z(\mathfrak{gl}(n+1))$ of $U$. This formula is proven by relating $\rho$ to the corresponding Harish-Chandra isomorphisms, and, alternatively, by using a purely computational approach.  Furthermore, we define a homomorphism from  $\mathcal D'(n) \otimes U(\mathfrak{gl}(n))$ to an algebra containing  $U$ as a subalgebra, so that $\sigma (\rho (u))  - u \in G_1 U$, for all $u \in U$, where $G_1 = \sum_{i=0}^{n} E_{ii}$. \\


\end{abstract}
\maketitle

\section{Introduction}
Capelli-type determinants are a powerful tool in invariant theory. One fundamental result is that the Capelli determinant $C_N(T)$ corresponding (up to a diagonal shift) to the $N \times N$ matrix ${\bf E}$ whose $(i,j)$th entry is the elementary matrix $E_{ij}$ is a polynomial in $T$ whose coefficients $C_k$ are central elements in the universal enveloping algebra $U(\mathfrak{gl}(N))$ of the Lie algebra $\mathfrak{gl}(N)$. Moreover, the coefficients $C_k$  are generators of the center $Z(\mathfrak{gl}(N))$ of $U(\mathfrak{gl}(N))$, and are usually referred as Capelli generators. On the other hand the elements $G_k = \tr ({\bf E}^k)$ for $k=1, \dots ,N$ also form a system of generators of  $Z(\mathfrak{gl}(N))$, sometimes known as the Gelfand generators. There is a nice transition formula between the Gelfand generators and the Capelli generators, \cite{U}, and this formula can be considered as a noncommutative version of the Newton identities. The applications of Capelli-type determinants extend well beyond 
relations and properties of elements in  $Z(\mathfrak{gl}(N))$. There are direct applications to classical invariant theory (see for example \cite{H}), and a more general treatment of the subject using  the theory of Yangians (see \cite{M} and the references therein). 

In this paper we relate Capelli-type determinants to another classical construction in representation theory - the mixed tensor type modules. These modules are modules over tensor products of an algebra of differential operators and a universal enveloping algebra  $U(\mathfrak{gl}(n))$. Modules of mixed tensor type, also known as tensor modules, or modules of Shen and Larson (see \cite{Sh} and \cite{L}),  can be defined over the Lie algebras $\mathfrak{sl}(n+1)$ after considering a suitable homomorphism. In this paper we define a $\mathfrak{gl}(n+1)$-version of this homomorphism, namely a map $\rho :U(\mathfrak{gl}(n+1)) \to \mathcal D' (n) \otimes U(\mathfrak{gl}(n))$, where $\mathcal D' (n)$ is the algebra of polynomial differential operators on $\mathbb C [t_0^{\pm 1}, t_1, \dots ,t_n]$ generated by $t_i/t_0$, $t_0 \partial_j$, $i>0$, $j \geq 0$. One of the main results of the paper is an explicit description of the image of the Capelli determinant $C_{n+1}(T)$ of $\mathfrak{gl}(n+1)$ under $\rho$. It turns out that the result is especially pleasant: up to a shift, $\rho (C_{n+1}(T))$ is a product of $C_{n}(T)$ and a linear factor. This leads to  explicit formulas for the images $\rho (G_k^{\mathfrak{gl}(n+1)})$ of the Gelfand generators  $G_k^{\mathfrak{gl}(n+1)}$ for $\mathfrak{gl}(n+1)$ under $\rho$. As a corollary, we also obtain the mixed tensor version of the noncommutative Cayley-Hamilton identities; see Corollary \ref{cor-ch}. To prove the formula for $\rho (C_{n+1}(T))$ we relate $\rho$ to the Harish-Chandra isomorphisms of $\mathfrak{gl}(n)$ and $\mathfrak{gl}(n+1)$. We expect that the explicit formulas for $\rho (C_{n+1}(T))$ and $\rho (G_k^{\mathfrak{gl}(n+1)})$ to have applications in the representation theory of the tensor modules, as one can easily and explicitly compute central characters.

Another result of the paper is finding a pseudo left inverse $\sigma$ of  $\rho$ in the following sense. The homomorphism $\sigma$ maps  $\mathcal D' (n) \otimes U(\mathfrak{gl}(n))$ to an algebra $U''$ that contains $U = U(\mathfrak{gl}(n+1))$ as a subalgebra is such that $\sigma (\rho (u)) - u \in G_1 U$ for every $u\in U$, where $G_1 = G_1^{\mathfrak{gl}(n+1)}= \sum_{i=0}^{n}E_{ii}$. In particular, we obtain that the kernel of $\rho$ is the ideal $(G_1)$ in $U$, while the kernel of the $\mathfrak{sl}(n+1)$ version of $\rho$ is trivial.

It is worth noting that the results concerning the images of the generators of $ Z(\mathfrak{gl}(n+1))$ under $\rho$ can be obtained with long and  direct computation. For this (alternative) approach we prove some more general identities that are included in the Appendix of this paper. We believe that these identities may be of independent interest. 

The organization of the paper is as follows. In Section 3 we define the homomorphism $\rho$. In Section 4 we relate $\rho$ with the corresponding Harish-Chandra isomoprhisms. The results of Section 4 are applied in the next section where the image of the Capelli-type determinant under $\rho$ is computed. Using the latter results we derive the formulas for the images of the Gelfand generators under $\rho$.  In Section 6 we define a pseudo-left inverse of $\rho$ and find the kernel of $\rho$. The Appendix contains some explicit formulas for the images under $\rho$ of a set of homogeneous elements of $U(\mathfrak{gl}(n+1))$. All the Gelfand generators for $\mathfrak{gl}(n+1)$ can be written as sums of elements from this set, so we obtain an alternative proof of the results in Sections 4 and 5.

\medskip
\noindent \textbf{Acknowledgements:} The first author is partly supported by  Simons Collaboration Grant 358245. He also would like to thank the Max Planck Institute in Bonn (where part of this work was completed) for the excellent working conditions. The second author thanks his family for their support. We thank Vera Serganova for the useful discussion.

\section{Notation and conventions}
Our base field is $\mathbb C$. All vector spaces, tensor products, and associative algebras will be considered over $\mathbb C$ unless otherwise specified. By $\delta_{kl}$  we denote the Kronecker delta function, which equals $1$ if $k=l$ and $0$ otherwise. Throughout the paper we fix a positive integer $n$.

For a Lie algebra $\mathfrak a$, by $U(\mathfrak a)$ we denote the universal enveloping algebra of $\mathfrak a$ and by $Z(\mathfrak a)$ the center of $U(\mathfrak a)$. By $\mathfrak{gl} (N)$ (respectively, $\mathfrak{sl} (N)$) we denote the Lie algebra of all (respectively, traceless) $N \times N$ matrices. We write $I_N$ for the identity matrix of $\mathfrak{gl}(N)$. For an $N\times N$ matrix $A$, the entries $A_{ij}$ will be indexed by  $1 \leq i,j \leq n$ if $N=n$, and by $0 \leq i,j \leq n$ for $N=n+1$.
In the latter case, we will refer to the top left entry as the $(0,0)$th entry. Similarly, the weights of $\mathfrak{gl} (n+1)$ will be written as $(n+1)$-tuples $(\mu_0, \dots ,\mu_{n})$, while those of $\mathfrak{gl} (n)$ as $n$-tuples $(\mu_1, \dots ,\mu_{n})$.

We will write  $E_{ij}$ for the $(i,j)$th elementary matrix of $\mathfrak{gl}(N)$, and  ${\bf E}_N$ for the $N \times N$ matrix whose $(i,j)$th entry is $E_{ij}$. Henceforth, we fix the Borel subalgebra $\mathfrak b_N$ and the Cartan subalgebra $\mathfrak h_N$ of $\mathfrak{gl}(N)$ to be the ones spanned by $E_{ij}$ ($i \leq j$) and $E_{kk}$ (all $k$), respectively. By $\mathfrak n_{\mathfrak{gl}(N)}^{+}$ (respectively, $\mathfrak n_{\mathfrak{gl}(N)}^{-}$) we denote the nilradical (respectively, the opposite nilradical) of $\mathfrak b_N$. In particular, $\mathfrak b_N = \mathfrak h \oplus \mathfrak n_{\mathfrak{gl}(N)}^{+}$ and 
$\mathfrak{gl}(N) = \mathfrak b_N \oplus \mathfrak n_{\mathfrak{gl}(N)}^{-}$. We will use these conventions both for $N=n$ and $N=n+1$.

By default, determinants will be column determinants. Namely, for an $n \times n$ matrix $A$ with entries $A_{ij}$ in an associative algebra,
$$
\det (A) := \sum_{\sigma \in S_n} A_{\sigma(1)1}A_{\sigma(2)2} \dots A_{\sigma(n)n},
$$
where $S_n$ is the $n$th symmetric group. The determinant of an $(n+1) \times (n+1)$ matrix is defined analogously.

For a square matrix $A$ and variable or a constant $v$, the expression $A+v$ (and $v+A$) should be understood as the sum of $A$ and the scalar matrix of the same size as $A$ having $v$ on the diagonal.

By ${\mathcal D (n)}$ we denote the algebra of polynomial differential operators on $\mathbb C^n$. The algebra ${\mathcal D' (n)}$ is the subalgebra of differential operators on ${\mathbb C} [t_0^{\pm 1}, t_1, \dots ,t_n]$ generated by $\frac{t_i}{t_0}$ for $i=1, \dots ,n$ and $t_0\partial_j$ for $j=0, \dots ,n$. Here $\partial_a$ stands for $\frac{\partial}{\partial t_a}$. We set ${\mathcal E}:= \sum_{i=0}^n t_i \partial_{i}$. 

We finish this section with some conventions that we will use throughout the paper. We assume that $\sum_{i=r}^{s} x_i = 0$ whenever $r>s$. For a subset $S$ of a ring $R$, by $(S)$ we denote the two-sided ideal of $R$ generated by $S$. For associative unital algebras $\mathcal A_1$ and  $\mathcal A_2$ and elements $a_i \in \mathcal A_i$, we often write $a_1$ and $a_2$ for the elements $a_1\otimes 1$ and $1 \otimes a_2$, respectively.

\section{The homomorphism $\rho$} \label{sec-rho}
In this section we define the homomorphism $\rho: U(\mathfrak{gl} (n+1)) \to \mathcal{D}' (n) \otimes U (\mathfrak{gl} (n))$ that  plays important role in this paper. This homomorphism can be considered as the $\mathfrak{gl} (n+1)$-version of a homomorphism $\rho_s: U(\mathfrak{sl} (n+1)) \to \mathcal{D}' (n) \otimes U (\mathfrak{gl} (n))$, which can be defined as follows, in brief. For a finite-dimensional $\mathfrak{gl}(n)$-module $V_0$ consider the corresponding trivial vector bundle $\mathcal V_0$ on $\mathbb P^n$. Then there is a natural map from $\mathfrak{sl}(n+1)$ to the algebra of differential operators on the space of sections of $\mathcal V_0$ over the open subset $\mathcal U_0 = \{t_0\neq 0\}$ of $\mathbb P^n$. This map leads to the homomorphism $\rho_s$. For details we refer the reader to, for example, \S 2 of \cite{GS}. The explicit formulas for $\rho_s$ (given in local coordinates) are listed in the proof of \cite[Lemma 2.3]{GS}.

Introduce the following elements of $\mathcal{D}' (n) \otimes U (\mathfrak{gl} (n))$:
$$
R_1 = - \frac{1}{n+1} \left( \mathcal{E} \otimes 1 + 1 \otimes G_1^{\mathfrak{gl} (n)}\right), \; R_2 = (\mathcal{E}+ n)\otimes 1,
$$
where $G_1^{\mathfrak{gl} (n)} = \sum_{i=1}^nE_{ii}$.
Note that $R_1$ and $R_2$ are central in $\mathcal{D}' (n) \otimes U (\mathfrak{gl} (n))$.

We next introduce three natural homomorphisms: the natural embedding $\iota_s : U(\mathfrak{sl}(n+1)) \to U(\mathfrak{gl}(n+1))$,  $\pi_g:  U(\mathfrak{gl}(n+1)) \to U(\mathfrak{sl}(n+1)) $, and  $\iota_g :  U(\mathfrak{gl}(n)) \to U(\mathfrak{sl}(n+1))$, as follows:
$$
\pi_g(B) = B - \frac{1}{n+1} \tr(B)I_{n+1},\;  \iota_g (C) = C - \tr(C) E_{00}.
$$
Define $\rho = \rho_s \pi_g$. The fact that $\rho_s$ is a homomorphism and the explicit formulas for $\rho_s$ imply the following.

\begin{lemma} 
The following correspondence 
\begin{eqnarray*}
E_{ab} &\mapsto& t_a \partial_b \otimes 1 + 1 \otimes E_{ab} + \delta_{ab}R_1 \; \mbox{  for }a,b>0,\\
E_{a0}  &\mapsto&  t_a \partial_0 \otimes 1 - \sum_{i>0} \frac{t_i}{t_0} \otimes E_{ai} \; \mbox{  for }a>0,\\
E_{0b} &\mapsto&  t_0 \partial_b \otimes 1 \;  \mbox{  for }b>0,\\
E_{00} &\mapsto& t_0 \partial_0 \otimes 1 + R_1.
\end{eqnarray*}
extends to the homomorphism $\rho: U(\mathfrak{gl} (n+1)) \to \mathcal{D}' (n) \otimes U (\mathfrak{gl} (n))$ of associative unital algebras.
\end{lemma}
 
We note that $\rho(G_1^{\mathfrak{gl}(n+1)})=0$ and  $\rho\iota_s = \rho_s$. Let us also define $$
\gamma: U(\mathfrak{sl} (n+1)) \to \mathcal{D}' (n) \otimes U (\mathfrak{sl} (n+1))
$$
by the identity $\gamma = (1\otimes \iota_g)\rho \iota_s$. We finish this section by collecting the identities for the homomorphisms that we introduced in this section. 

\begin{proposition} We have $\gamma = (1 \otimes \iota_g) \rho_s$, $\pi_g \iota_s = {\rm Id}$, and $\rho_s \pi_g = \rho$, and all other relations that directly follow from these three; in that sense, the following diagram is commutative.
$$
\xymatrix{ U(\mathfrak{gl}(n+1)) \ar@<0.5ex>[d]^{ \pi_g}   \ar[r]^{\rho\hspace{0.5cm}}  &  \mathcal{D}'(n) \otimes U(\mathfrak{gl}(n))   \ar@<0.5ex>[d] ^{1\otimes \iota_g}\\
U(\mathfrak{sl}(n+1)) \ar[u]^{\iota_s}    \ar[r]^{\gamma  \hspace{0.8cm}}   \ar[ur]^{\rho_s}  &  \mathcal{D}'(n) \otimes U(\mathfrak{sl}(n+1))     }
$$

 \end{proposition}

\section{Images under Harish-Chandra isomorphisms} \label{sec-5}

In this section we relate the restriction of $\rho$ on the center $Z(\mathfrak{gl} (n+1))$ of $U(\mathfrak{gl} (n+1))$ with the Harish-Chandra isomorphisms. 

We first recall the definition of the Harish-Chandra isomorphism and define an isomorphism of Harish-Chandra type with domain ${\mathbb C}[\mathcal E] \otimes  Z(\mathfrak{gl} (n))$.

For a weight $\lambda = (\lambda_1, \dots ,\lambda_n)$ of $\mathfrak{gl} (n)$ denote by $M_n(\lambda)$ and $L_n(\lambda)$ the Verma module of highest weight $\lambda$ and its unique simple quotient, respectively. We similarly define the $\mathfrak{gl} (n+1)$-modules $M_{n+1}(\mu)$ and $L_{n+1}(\mu)$ for a weight $\mu = (\mu_0,\mu_1, \dots ,\mu_n)$ of $\mathfrak{gl}(n+1)$. In particular, if $\lambda = (\lambda_1, \dots ,\lambda_n)$, then there is a (highest weight) vector $v_0$ of $M_n(\lambda)$ such that $M_n(\lambda) = U(\mathfrak n_{\mathfrak{gl}(n)}^-) \otimes {\mathbb C}v_0$ as vector spaces, $E_{ab}v_0 = 0$ for $1\leq a <b \leq n$, and $E_{ii}v_0 = \lambda_i v_0$ for $i = 1, \dots ,n$. 

Henceforth we set  $\delta_N = (0,-1, \dots ,-N+1)$. If $(b_1,\dots,b_N) \in {\mathbb C^N}$, then the evaluation homomorphism ${\rm ev}_{b_1,\dots,b_N} : {\mathbb C}[x_1,\dots,x_N] \to \mathbb C$ is defined by ${\rm ev}_{b_1,\dots,b_N}(p) = p(b_1,\dots,b_N)$. Every $z' \in Z(\mathfrak{gl} (n))$ acts on $L(\lambda)$ as $\chi_{\lambda}(z') \mbox{Id}$, where $\chi_{\lambda} (z') = \mbox{ev}_{\lambda + \delta_n} (\chi_n(z'))$ and $\chi_n : Z(\mathfrak{gl} (n)) \to  {\mathbb C}[\ell_1,\dots,\ell_n]^{S_{n}}$ is the Harish-Chandra isomorphism. We similarly  define $\chi_{n+1}: Z(\mathfrak{gl} (n+1)) \to  {\mathbb C}[\ell_0, \dots ,\ell_n]^{S_{n+1}}$ using the action of any element of $z \in Z(\mathfrak{gl} (n+1))$ on a simple highest weight module of $\mathfrak{gl} (n+1)$.

Next, define $\chi_{0,n} : {\mathbb C}[\mathcal E] \otimes  Z(\mathfrak{gl} (n)) \to {\mathbb C}[\ell_0] \otimes {\mathbb C}[\ell_1, \dots ,\ell_n]^{S_{n}}$ by  $\chi_{0,n} \left(\sum_{i}\mathcal E^i \otimes z_i \right) = \sum_{i} \ell_0^i \otimes \chi_n(z_i)$. Note that $\chi_{0,n}$ is an isomorphism and  that $\ell = \chi_{0,n}(R_1)$, where $\ell := \\ -\frac{1}{n+1} \left( \ell_0 + \sum_{i=1}^n\ell_i+\frac{n(n-1)}{2}\right)$. The latter follows from the fact that 
$\chi_n(G_1^{\mathfrak{gl}(n)}) = \sum_{i=1}^n\ell_i+\frac{n(n-1)}{2}$ (see Example 7.3.4 in \cite{M}).

Recall that a weight $\mu$ of $\mathfrak{gl}(n+1)$ is \emph{antidominant} if $\mu_i - \mu_{i+1} \notin {\mathbb Z}_{\geq 0}$ for all $i = 0,\dots ,n$. A well-known fact is that $M_{n+1} (\mu)$ is simple if and only if $\mu$ is anti-dominant. Also, by a Theorem of Duflo, the annihilator of a simple Verma module  $M_{n+1} (\mu)$ is generated by $z-\chi_{\mu}(z)$, $z \in Z(\mathfrak{gl}(n+1))$; see for example \S 8.4.3 in \cite{Dix}. 

Let 
$$
\mathcal F_a = \mbox{Span} \{t_0^{a-k_1-\cdots -k_n}t_1^{k_1}\dots t_n^{k_n}\mid  k_1,\dots , k_n \in \mathbb Z_{\geq 0} \}
$$
and consider $\mathcal F_a$ as a  ${\mathcal D}'(n)$-module.  Note that  $\mathcal F_a = {\mathcal D}'(n)(t_0^a)$ and that $\mathcal E = a \mbox{Id}$ on $\mathcal F_a$.

\begin{lemma} \label{lem-faithful}
Let $\mathcal{AW}_n$ denote the set of all anti-dominant $\mathfrak{gl}(n)$-weights $\lambda$ such that $\lambda_1 \notin {\mathbb Z}$.  Then the modules $\bigoplus_{a \in \mathbb{Z}} \mathcal{F}_a$  and $\bigoplus_{\lambda \in \mathcal{AW}_n} M_n(\lambda)$  are faithful over $\mathcal{D}^{'}(n)$ and $U(\mathfrak{gl}(n))$, respectively.
\end{lemma}

\begin{proof}
The fact that $\bigoplus_{\lambda \in \mathcal{AW}_n} M_n(\lambda)$ is faithful follows from the Theorem of Duflo and the fact that the annihilator of the direct sum of modules is the intersection of their annihilators.

We now prove that the annihilator of $\bigoplus_{a \in \mathbb{Z}} \mathcal{F}_a$ is trivial. Suppose for the sake of contradiction that $x \in \mathcal{D}^{'}(n)$ annihilates $\bigoplus_{a \in \mathbb{Z}} \mathcal{F}_a$. Next, choose a monomial $x_m = t_0^{a_0} \dots t_n^{a_n} \partial_0 ^ {b_0} \dots \partial_n ^{b_n}$ in the sum expansion of $x$ that is $\partial$-lexicographically maximal (i.e. relative to  the $\partial_0$-degree, $\partial_1$-degree, $\dots$, $\partial_n$-degree) among all monomials in $x$. Then, as $e_0, \dots ,e_n$ vary in a suitable set, the coefficient of $t_0^{e_0+a_0-b_0} \dots t_n^{e_n+a_n-b_n}$ in $x(t_0^{e_0}\dots t_n^{e_n})$ is equal to some polynomial in $e_0, \dots , e_n$. Furthermore, the leximaximal property of $t_0^{a_0} \dots t_n^{a_n} \partial_0 ^ {b_0} \dots \partial_n ^{b_n}$ yields that the coefficient of $e_0^{b_0}\dots e_n^{b_n}$ in this polynomial is nonzero. Thus this polynomial is nonzero, and for suitable $e_0, \dots, e_n$ it evaluates to a nonzero number, contradicting the fact that $x$ annihilates $\mathcal{D}^{'} (n)$. 
\end{proof}
\begin{theorem} \label{thm-hc-rho}
Let $\tau$ be the endomorphism on ${\mathbb C}[\ell_0,\ell_1,\dots ,\ell_n]$ defined by $\tau (p(\ell_0,\ell_1,\dots,\ell_n)) = p(\ell_0+ \ell, \ell_1+ \ell -1,\dots,\ell_n+\ell-1)$. Then we have that 
$$
\rho|_{Z(\mathfrak{gl} (n+1))} = \chi_{0,n}^{-1} \tau \chi_{n+1}.
$$
In particular, $\rho \left( Z(\mathfrak{gl} (n+1)) \right) $ is a subalgebra of ${\mathbb C}[\mathcal E] \otimes  Z(\mathfrak{gl} (n))$, and the following diagram  is commutative:
\begin{equation}
\begin{CD}
Z(\mathfrak{gl} (n+1)) @>\rho>>{\mathbb C}[\mathcal E] \otimes  Z(\mathfrak{gl} (n))\\
@V \chi_{n+1} VV @VV \chi_{0,n} V\\
{\mathbb C}[\ell_0,\ell_1,\dots ,\ell_n]^{S_{n+1}} @>\tau>>{\mathbb C}[\ell_0] \otimes {\mathbb C}[\ell_1,\dots,\ell_n]^{S_{n}}
\end{CD}\notag
\end{equation}

\medskip
\end{theorem}
\begin{proof}
Let $ \lambda = (\lambda_1,\dots ,\lambda_n)$ be in $\mathcal{AW}_n$ and $a \in {\mathbb Z}$. Also, let $M(a,\lambda) = {\mathcal F}_a \otimes M_n (\lambda)$. We first note that, in order to prove the identity in the theorem, it is sufficient to check that $\rho (z) = \chi_{0,n}^{-1} \tau \chi_{n+1} (z)$ for all $z \in Z(\mathfrak{gl} (n+1))$, as an identity of endomorphisms of $M(a,\lambda)$. Indeed, by Lemma  \ref{lem-faithful} and by the fact that the tensor product of faithful modules is a faithful module (see, for example, \cite{Berg}), the module $\bigoplus_{a \in \mathbb Z, \lambda \in {\mathcal{AW}}_n} M(a,\lambda)$ is a faithful module over $\mathcal D'(n) \otimes U(\mathfrak{gl} (n))$.

We next observe that if $M(a,\lambda)$ is considered as a $\mathfrak{gl} (n+1)$-module through $\rho$, then $M(a,\lambda) \cong M_{n+1}(\tilde{\lambda})$, where $\tilde{\lambda} = (a+ r_1, \lambda_1 + r_1,\dots ,\lambda_n + r_1)$ and $r_1 = -\frac{1}{n+1} \left( a + \sum_{i=1}^n \lambda_i \right)$. To prove this, let us fix  a highest weight vector $v_{\lambda}$ of $M_n(\lambda)$. Then it is straightforward to check that $E_{ab}(t_0^a \otimes v_{\lambda}) = 0$ for $a <b$ and that the weight of $t_0^a \otimes v_{\lambda}$ is $\tilde{\lambda}$. On the other hand, since $a \in \mathbb Z$ and $\lambda \in {\mathcal{AW}}_n$, $\tilde{\lambda}$ is anti-dominant (since we know $\lambda_1 \notin \mathbb{Z}$). Hence, $M_{n+1}(\tilde{\lambda})$ is simple. To conclude the proof of $M(a,\lambda) \cong M_{n+1}(\tilde{\lambda})$ we show that both modules have the same formal characters. Indeed, observe that for a monomial $u \in U(\mathfrak{n}_{\mathfrak{gl}(n)}^-)$,  the vector $t_0^a \left( t_1/t_0 \right)^{k_1}\dots \left( t_n/t_0 \right)^{k_n} \otimes uv_{\lambda}$ in $M(a,\lambda)$, and the vector $E_{10}^{k_1}\dots E_{n0}^{k_n} u v_{\tilde{\lambda}}$ in $M_{n+1}(\tilde{\lambda})$, have the same weights. Therefore,
$$
\mbox{ch} \left( {\mathcal F}_a \otimes M_n (\lambda) \right) = \mbox{ch} (t_0^a\mathbb C \left[t_1/t_0,\dots ,t_n/t_0\right] \otimes U(\mathfrak{n}_{\mathfrak{gl}(n)}^-) v_{\lambda}) = \mbox{ch} (U(\mathfrak{n}_{\mathfrak{gl}(n+1)}^-) v_{\tilde{\lambda}}).  
$$

Since $M(a,\lambda) \cong M_{n+1}(\tilde{\lambda})$, we have that 
 every $z \in Z(\mathfrak{gl} (n+1))$  acts on  $M(a,\lambda)$ as $\chi_{\tilde{\lambda}}(z) \mbox{Id}$. Recall that by definition,
$$
\chi_{\lambda} (z') = \mbox{ev}_{\lambda + \delta_n} (\chi_n(z')),\; \chi_{\tilde{\lambda}} (z) = \mbox{ev}_{\tilde{\lambda} + \delta_{n+1}}(\chi_{n+1}(z)). 
$$ 
for every  $z' \in Z(\mathfrak{gl} (n))$ and $z \in Z(\mathfrak{gl} (n+1))$.  

On the other hand, if $\xi = \chi_{0,n}^{-1} \tau \chi_{n+1}$, then $\xi(z)$ acts on $M(a,\lambda)$ as $\chi_{a,\lambda} (\xi(z)) \mbox{Id}$, where $\chi_{a,\lambda} \left(\sum_{i}\mathcal E^i \otimes z_i' \right) = \sum_{i} a^i \chi_{\lambda}(z_i')$, for $z_i' \in Z(\mathfrak{gl} (n))$. Let $p = \chi_{n+1}(z)$. It remains to prove that $\chi_{a,\lambda} \chi_{0,n}^{-1} \tau (p) = \mbox{ev}_{\tilde{\lambda} + \delta_{n+1}}p$.
This follows from the fact that 
$\chi_{a,\lambda} ({\mathcal E}^i \otimes z')  = \mbox{ev}_{a,\lambda + \delta_n} (\chi_{0,n} ({\mathcal E}^i \otimes z'))$ for every $z' \in Z(\mathfrak{gl} (n))$ and that $\mbox{ev}_{a,\lambda + \delta_n} \tau = \mbox{ev}_{\tilde{\lambda} + \delta_{n+1}}$.
\end{proof}

\section{Capelli-type determinants} \label{sec-6}
For any formal variable $T$, we define the Capelli determinant of $\mathfrak{gl}(n+1)$ and $\mathfrak{gl}(n)$ as follows:

$$
C_{n+1}(T) = \det \left[ \begin{matrix} E_{00} - T  & E_{01} & \cdots & E_{0n} \\ E_{10} & E_{11} - T-1 & \cdots & E_{1n} \\
\vdots & \vdots & \ddots  & \vdots\\
 E_{n0} & E_{n1} & \cdots & E_{nn} - T - n 
\end{matrix} \right] 
$$

and

$$
C_n(T) = \det \left[ \begin{matrix} E_{11} - T  & E_{12} & \cdots & E_{1n} \\ E_{21} & E_{22} - T-1 & \cdots & E_{2n} \\
\vdots & \vdots & \ddots  & \vdots\\
 E_{n1} & E_{n2} & \cdots & E_{nn} - T - n +1
\end{matrix} \right], 
$$
respectively. Note that $T$ is a formal variable that commutes with all $E_{ij}$, and $C_n(T)$ and $C_{n+1}(T)$ will be treated as polynomials in this formal variable.  We also note that the polynomials $C_{n+1} (T)$ appear in \cite{M} and \cite{U} with a slight change. Namely, the polynomial $C_{n+1}^{M}(T)$ defined in \S 7.1 of \cite{M} (also called Capelli determinant) is related to our $C_{n+1} (T)$ via the identity $C_{n+1} (T) = C_{n+1}^M (-T)$. On the other hand, the polynomial $C_{n+1}^U(T)$ defined in \cite{U} satisfies the relation $C_{n+1} (T) = C_{n+1}^U (T+n)$.

Define $C_{\rho} (T)$ via the identity $C_{\rho} (\rho(T)) = \rho(C_{n+1} (T))$. For convenience we will write $T$ for $\rho(T)$.
\begin{theorem} \label{rho-capelli}
The following identity holds:
$$
C_{\rho}(T+R_1) = (\mathcal{E}-T)C_n(T+1).
$$
In particular, $C_{\rho} (\mathcal{E}+R_1) = 0$. 
\end{theorem}
\begin{proof} The identity is equivalent to the following:
$$
\rho(C_{n+1}^M (T)) = (\mathcal{E}+T+R_1)C_n^M(T+R_1-1).
$$
To prove the latter we  use that
$$
\chi_n (C_{n}^M (T)) = (T+\ell_1)\dots (T+\ell_n).
$$
(Theorem 7.1.1 in \cite{M}) and the corresponding formula for $\chi_{n+1} (C_{n+1}^M (T))$. To complete the proof we use Theorem \ref{thm-hc-rho} and that $\chi_{0,n}(R_1) = \ell$.
\end{proof}

\begin{example}
In the case $n=2$, the identity in Theorem \ref{rho-capelli} is the following:
\begin{eqnarray*}
&& \det \left[ \begin{matrix} t_0\partial_0 - T   & t_0 \partial_1  & t_0\partial_2 \\
 t_1 \partial_0 - \frac{t_1}{t_0}E_{11} -\frac{t_2}{t_0}E_{12} &  t_1 \partial_1+E_{11}  - T - 1 & t_1\partial_2+E_{12} \\
t_2 \partial_0 - \frac{t_1}{t_0}E_{21} -\frac{t_2}{t_0}E_{22} &  t_2 \partial_1 + E_{21} & t_2\partial_2+E_{22}-T-2
 
\end{matrix} \right]
 \\& = &  (\mathcal E - T)C_2(T+1) = (\mathcal{E} - T) \det \left[ \begin{matrix} E_{11}-T-1   &  E_{12}  \\
E_{21} &  E_{22}-T-2 
\end{matrix} \right].
\end{eqnarray*}
\end{example}

\begin{corollary}\label{cor-ch} The following identities hold:
$$ C_{\rho} ({\bf E}_n+R_1-n) = 0, \; C_{\rho} ({\bf E}_n^t-1+R_1) = 0$$
\end{corollary}
\begin{proof}

The identities follow from Theorem \ref{rho-capelli} and the noncommutative version of the Cayley-Hamilton theorem:
$$
C_n^M(-{\bf E}_n+n-1) = C_n^M(-{\bf E}_n^t) = 0. 
$$
(This is Theorem 7.2.1 in \cite{M}.)
\end{proof}

We conclude this section by giving explicit formulas of the images of the Gelfand invariants of $Z(\mathfrak{gl}(n+1))$ under $\rho$. To define these invariants we first introduce some special elements in $U(\mathfrak{gl}(n+1))$. Set $r_0^{\mathfrak{gl}(n+1)} (a,b) = \delta_{ab}$, and let
\begin{equation} \label{def-r-k}
r_{k+1}^{\mathfrak{gl}(n+1)} (a,b) = \sum_{i_1,\dots,i_k} E_{ai_1}E_{i_1i_2}\dots E_{i_kb}
\end{equation}
where the sum runs over all (not necessarily distinct) $0\leq i_1,\dots,i_k \leq n$. Then $$G_k^{{\mathfrak{gl}(n+1)}} = \sum_{i=0}^n r_k^{{\mathfrak{gl}(n+1)}} (i,i)$$ is the \emph{Gelfand invariant of degree $k$ of  $\mathfrak{gl} (n+1)$}. In other words, $G_k^{\mathfrak{gl}(n+1)} = \tr ({\bf E}_{n+1}^{\,k})$. We define $r_k^{\mathfrak{gl}(n)}(a,b)$ and $G_k^{\mathfrak{gl}(n)}$ for $\mathfrak{gl}(n)$ analogously. It is well-known fact that $Z(\mathfrak{gl} (n+1))$ is a polynomial algebra in  $G_k^{{\mathfrak{gl}(n+1)}}$ for $k=1,2,\dots,n+1$.

\begin{theorem} \label{thm-rho-g}The following formula holds for all positive integers $k$:

\begin{eqnarray*}
\rho(G_k^{\mathfrak{gl}(n+1)})& = & \left( \sum_{g=0}^{k-1} \binom{k}{g} R_1^g R_2^{k-1-g} \right)\left( \mathcal{E} \otimes 1 \right) + (n+1) R_1^k + \sum_{g=0}^{k-1} \binom{k}{g} R_1^g \left( 1 \otimes G_{k-g}^{\mathfrak{gl}(n)} \right)\\ 
&& - \sum_{m=2}^k \left( \sum_{g=0}^{k-m} \binom{k}{g} R_1^g R_2^{k-m-g} \right)\left( 1 \otimes G_{m-1}^{\mathfrak{gl}(n)} \right).
\end{eqnarray*}
\end{theorem}
\begin{proof}
Let us first recall the following identity:
$$
\frac{C_{N}(T-N+1) - C_{N}(T-N)}{C_N(T-N+1)} = \sum_{k=0}^{\infty}{G_k^{\mathfrak{gl}(N)} T^{-1-k} }. 
$$
This is sometimes referred as the noncommutative analogue of the classical Newton formula; see Theorem 4 in \cite{U} (see also Theorem 7.1.3 in \cite{M}). Note that this identity should be considered over the ring of formal Laurent series with coefficients in $Z(\mathfrak{gl}(N))$.

On the other hand, Theorem \ref{rho-capelli} implies that 
\begin{equation} \label{eq-rho-newton}
    \frac{C_{\rho} (T-n-1)}{C_{\rho}(T - n)} = \left( 1 - \frac{1}{T-R_1-R_2}\right)\frac{C_n (T-R_1-n)}{C_n (T- R_1-n+1)}
\end{equation}
in $\mathbb C[\mathcal E] \otimes Z(\mathfrak{gl} (n))$. Applying $\rho$ to the Newton formula for $N=n+1$, we obtain
$$
 \frac{C_{\rho} (T-n-1)}{C_{\rho}(T-n)} = 1 - \sum_{k\geq 0} \rho (G_k^{\mathfrak{gl}(n+1)}) T^{-1-k} = 1 - (n+1)T^{-1} - T^{-1}\sum_{k\geq 1} \rho (G_k^{\mathfrak{gl}(n+1)}) T^{-k}.
$$
(recall that we identify $\rho(T)$ with $T$). Analogously, we can express the right hand side of \eqref{eq-rho-newton} as a power series in $T^{-1}$. We complete the proof by comparing and computing the coefficients of $T^{-k-1}$ on both sides. \end{proof}

\begin{example}
Theorem \ref{thm-rho-g} applied for $k=3$ gives the following formula:
\begin{eqnarray*}
\rho(G_3^{\mathfrak{gl}(n+1)})&=&(R_2^2+3R_1R_2+3R_1^2)(\mathcal{E} \otimes 1)+(n+1)R_1^3+(1\otimes G_3^{\mathfrak{gl}(n)}) + 3R_1(1\otimes G_2^{\mathfrak{gl}(n)}) \\
&& + 3R_1^2(1\otimes G_1^{\mathfrak{gl}(n)})-(1\otimes G_2^{\mathfrak{gl}(n)})-(R_2+3R_1)(1\otimes G_1^{\mathfrak{gl}(n)}).
\end{eqnarray*}
\end{example}

\section{Pseudo left inverse of $\rho$}

In this subsection we prove that the kernel of $\rho : U(\mathfrak{gl} (n+1)) \to \mathcal{D}'(n) \otimes U (\mathfrak{gl} (n))$ is $\left( G_1 \right)$ and define a family of homomorphisms $\sigma$, such that $\sigma(\rho (t)) - t \in G_1 U$ for all $t \in U$. Here and henceforth $U= U (\mathfrak{gl} (n+1))$ and $G_1 = G_1^{\mathfrak{gl}(n+1)} = \sum_{i=0}^{n} E_{ii}$. The map $\sigma$ is pseudo left inverse of $\rho$ in the sense that $\sigma\rho = \mbox{Id} \mod (G_1)$.

We first define the domain of $\sigma$. From now on, we set for simplicity $U= U (\mathfrak{gl} (n+1))$. Let $U'$ be the extension of $U$ defined by $U'= \\ U\langle X \rangle/\left( [U, X], C_{n+1}(X)\right) $. Equivalently, we define $U'$ by  considering a trivial central extension $\mathfrak{g}_{n+1}^X = \mathfrak{gl} (n+1) \oplus {\mathbb C}X$ of $\mathfrak{gl} (n+1)$, and letting  $U' = U(\mathfrak{g}_{n+1}^X)/ \left(  C_{n+1}(X) \right)$.
The following lemma is standard.
\begin{lemma} \label{unique-rep} 
Every element in $U'$ can be written uniquely in the form $u = \sum_{i=0}^n w_i X^i$ for some $w_i \in U$. In particular, $C_n(X)$ and $C_n(X+1)$ are nonzero in $U'$. \end{lemma}

\begin{lemma}
$U'$ is a domain.
\end{lemma}
\begin{proof}
Let $U[X] =  U(\mathfrak{g}_{n+1}^X)$.  To prove that the ring $U' = U[X]/ \left(  C_{n+1}(X) \right)$ is a domain, it is enough to show that  the associated graded ring ${\rm gr}\, U'$ is a domain.  Let $I= \left(  C_{n+1}(X) \right)$. Then  we have that ${\rm gr}\, U' = {\rm gr}\, U[X] / {\rm gr}\, I$ (see for example, \S 2.3.10 in \cite{Dix}). Since ${\rm gr}\, U[X]$ is a polynomial ring, and hence, a unique factorization domain, it is enough to show that the polynomial $C_{n+1}(-X)$ is irreducible in ${\rm gr}\, U[X]$. 

Let  $C_{n+1}(-X) = X^{n+1} + c_1 X^{n}+\cdots + c_{n+1}$. We know that the center  $ Z(\mathfrak{gl} (n+1))$ of $ U(\mathfrak{gl} (n+1))$ is the polynomial algebra $\mathbb C [c_1,\dots ,c_{n+1}]$. Assume that $C_{n+1}(-X) = (X^k + \cdots + a_{k-1}X + a_k)(X^{m} + \cdots + b_{m-1}X + b_{m})$ for some $a_i,b_i$ in ${\rm gr}\, U$ of graded degree $i$. In particular, $a_k b_{m} = c_{n+1}$.
Note that $c_{n+1}$ is the determinant of ${\bf E}$ and hence the ideal generated by $c_{n+1}$ in ${\rm gr}\, U$ is a determinantal ideal. It is well-known that any such ideal is a prime ideal (equivalently, that ${\rm gr}\, U/\left( c_{n+1}\right)$ is an irreducible variety); see, for example, Proposition 1.1 in \cite{BV}. Therefore, $c_{n+1}$ is irreducible in ${\rm gr}\, U$. We thus may assume that  $a_k$ is a constant, and that $b_{m}$ is a constant multiple of $c_{n+1}$. Then  $k = 0$ and $m = n+1$, and hence, $C_{n+1}(-X)$ is irreducible. \end{proof}

Since $U'$ is a quotient of the universal enveloping algebra of $\mathfrak g_{n+1}^X$, $U'$ is a  right (and left) Noetherian domain. Then by a theorem of Goldie,  \cite{Gol}, $U'$ is also a right Ore domain. Let $U''$ denote the right quotient ring of $U'$, i.e. the  right Ore localization of $U'$ relative to all nonzero elements of $U'$. We have natural embeddings  $\iota:U \to U'$ and $\iota':U' \to U''$ that will allow us to consider the elements of $U$ as elements of $U'$, and those of $U'$ as elements of $U''$.  Denote by $Y$ the left and right inverse of $C_n(X)$ in $U''$, i.e. $YC_n(X) = C_n(X)Y = 1$.

\begin{lemma} \label{lem-y-com}
$Y$ commutes with $X$ and $E_{ij}$ whenever $i,j>0$.
\end{lemma}
\begin{proof}
To prove the result we observe that $Y$ commutes with all elements that commute with $C_n(X)$.
\end{proof}
In order to define the  homomorphism  $\sigma: \mathcal{D}'(n) \otimes U (\mathfrak{gl} (n)) \to U''$, we next introduce some distinguished elements of $U''$. We treat again $T$ as a formal variable that commute with all $E_{ij}$. We first define the $n \times n$ matrix $M_n(T)$, whose $(i,j)$th entry is 
$$
M_n(T)_{ij} = \det \left[ \begin{matrix} E_{11} - T -1 & E_{12} & \cdots &0& \cdots& E_{1n} \\ E_{21} & E_{22} - T-2 & \cdots &0& \cdots& E_{2n} \\
\vdots & \vdots & \ddots & 0 & \ddots & \vdots\\
E_{i1} & E_{i2} & \cdots & 1 & \cdots & E_{in} \\
\vdots & \vdots & \ddots & 0 & \ddots & \vdots\\
 E_{n1} & E_{n2} & \cdots &0& \cdots& E_{nn} - T - n+1 
\end{matrix} \right] 
$$
where the entry of $1$ in the determinant above is the $(i,j)$th entry, while the $(k,k)$th entry equals $E_{kk}-T - k$ for $0 < k < j$ and $E_{kk}-T - k+1$ for $k > j$. 

Similarly, we define
$$
M_{n+1}(T)_{ij} = \det \left[ \begin{matrix} E_{00} - T -1 & E_{01} & \cdots &0& \cdots& E_{0n} \\ E_{10} & E_{11} - T-2 & \cdots &0& \cdots& E_{1n} \\
\vdots & \vdots & \ddots & 0 & \ddots & \vdots\\
E_{i0} & E_{i1} & \cdots & 1 & \cdots & E_{in} \\
\vdots & \vdots & \ddots & 0 & \ddots & \vdots\\
 E_{n0} & E_{n1} & \cdots &0& \cdots& E_{nn} - T - n 
\end{matrix} \right].
$$

The next lemma is standard but for reader's convenience we include a short proof. We will use the result both for $\mathfrak{gl}(n)$ and $\mathfrak{gl}(n+1)$.
\begin{lemma}
\label{matrix-comm}Let $V$ and $W$ be $n \times n$ matrices with entries in $U(\mathfrak{gl}(n))[T]$ and let $P$ be a nonzero element of $Z(\mathfrak{gl}(n))[T]$ such that $VW=PI_n$. Then $WV=PI_n$.
\end{lemma}

\begin{proof}
Since $P$ is central in $U(\mathfrak{gl}(n))[T]$, we can localize $U(\mathfrak{gl}(n))[T]$ relative to its multiplicative subset generated by $P$. Denote by $L$ the corresponding localized ring. Let $V'=P^{-1}V$. We then have $n \times n$ matrices $V'$ and $W$ with entries in $L$ such that $V'W=I_n$. We next use that  $U(\mathfrak{gl}(n))[T]$ is Noetherian, so $L$ is Noetherian, and hence the ring of $n \times n$ matrices with entries in $L$ is Noetherian as well. And since $V'W=I_n$ in the Noetherian ring $L$, we have $WV'=I_n$, \cite{Jac}. Thus $WV=PI_n$, as desired. \end{proof}

\begin{lemma}
\label{matrix-inv-n} The following identities hold:
$$({\bf E}_n^t-T)M_n(T)= M_n(T)({\bf E}_n^t-T)= C_{n}(T); \; M_n(T)^t ({\bf E}_n-T-n) = ({\bf E}_n-T-n) M_n(T)^t= C_{n}(T+1).$$
Also, $({\bf E}_{n+1}^t-T)M_{n+1}(T)= C_{n+1}(T)$, etc. \end{lemma}

\begin{proof}
First we show $({\bf E}_n^t-T)M_n(T)=C_n(T)$, or, equivalently, $\sum_k (E_{ki}-\delta_{ki}T) M_n(T)_{kj} = \delta_{ij}C_n(T)$. We adopt the reasoning used in Section 1 of \cite{U}. For this we work over the algebra $\Lambda_n \otimes U(\mathfrak{gl}(n))$, where  $\Lambda_n$ is   the exterior algebra with generators $e_1, \dots, e_n$.  Let $F_m=\sum_{l}E_{lm}e_l$. Then 
$$M_n(T)_{kj}e_1\dots e_n = (F_1 - (T+1)e_1) \dots e_k \dots (F_n - (T+n-1)e_n),
$$ where $e_k$ is the $j$th term in the product. We have that \begin{eqnarray*} \sum_k (E_{ki}-\delta_{ki}T) M_n(T)_{kj} e_1\dots e_n  & = & { (-1)^{j-1}\left(\sum_k(E_{ki}-\delta_{ki}T)e_k\right)(F_1 - (T+1)e_1) \dots} \\ &&   {F_{j-1}-(T+j-1)e_{j-1}) ((F_{j+1}-(T+j)e_j) \dots  (F_n - (T+n-1)e_n)} \\
& = &  (-1)^{j-1}(F_i-Te_i)(F_1 - (T+1)e_1) \dots  (F_{j-1}-(T+j-1)e_{j-1})\\ &&(F_{j+1}-(T+j)e_j) \dots (F_n - (T+n-1)e_n).
\end{eqnarray*} 

We next prove an identity analogous to Lemma 1 in \cite{U}. Specifically, we claim that 
\begin{equation}\label{second-paragraph}
(F_r-(T+l)e_r)(F_s-(T+l+1)e_s)=-(F_s-(T+l)e_s)(F_r-(T+l+1)e_r),
\end{equation}
or, equivalently, 
\begin{eqnarray*} 
\sum_{p,q} (E_{pr}e_p - (T+l)e_r)(E_{qs}e_q - (T+l+1)e_s) + \sum_{p,q} (E_{qs}e_q - (T+l)e_s)(E_{pr}e_p - (T+l+1)e_r) = 0.
\end{eqnarray*} 
{The left hand side equals  
\begin{eqnarray*} & \; & \sum_{p,q}[E_{pr},E_{qs}]e_pe_q + \sum_q E_{qs}e_re_q + \sum_p E_{pr}e_se_p \\
&&= \sum_{p,q} \delta_{qr} E_{ps}e_pe_q - \sum_{p,q} \delta_{ps}E_{qr}e_pe_q + \sum_q E_{qs}e_re_q + \sum_p E_{pr}e_se_p \\
&&= \left(\sum_p E_{ps}e_pe_r + \sum_q E_{qs} e_re_q\right) + \left(\sum_p E_{pr}e_se_p - \sum_q E_{qr}e_se_q\right) = 0,
\end{eqnarray*} 
as claimed.}
Now, applying \eqref{second-paragraph} repeatedly, we obtain $$\sum_k (E_{ki}-\delta_{ki}T) M_n(T)_{kj} e_1\dots e_n = (F_1 - Te_1) \dots (F_i-(T+j-1)e_i) \dots (F_n - (T+n-1)e_n),$$
where the $F_i-(T+j-1)e_i$ is the $j$th term. If $j=i$, this yields the desired. Otherwise, note that if $r=s$ in  \eqref{second-paragraph} then $(F_r-(T+l)e_r)(F_r-(T+l+1)e_r)=0$. Applying \eqref{second-paragraph} repeatedly and combining it with the last identity (considering separately the cases $i<j$ and $i>j$), we obtain
$$(F_1 - Te_1) \dots (F_i-(T+j-1)e_i) \dots (F_n - (T+n-1)e_n)=0.$$ Thus, $({\bf E}_n^t-T)M_n(T)=C_n(T)$, as desired.

The identity $M_n(T)^t(\mathbf{E}_n-T-n)=C_n(T+1)$ follows directly from Proposition 2 of \cite{U}. The remaining statements for the $n \times n$ matrices follow from Lemma \ref{matrix-comm}.  The proofs of the statements involving the $(n+1) \times (n+1)$ matrices are analogous.
\end{proof}

We next define some elements in $U''$ that will also be used to define the pseudo-left inverse $\sigma$ of $\rho$. Let $u_0=1$, and let $u_1,\dots ,u_n$ be defined via the matrix equation
$$
[u_1, u_2,\dots ,u_n] = -[E_{10}, E_{20},\dots ,E_{n0}]M_n(X)Y.
$$
Some properties of $u_i$ are listed in the following lemmas. 

\begin{lemma} \label{lemma-u-e-x} $\sum_{a=0}^n u_aE_{ia}=u_iX$  for $i=0,1,\dots ,n$.
\end{lemma}
\begin{proof} 
We have $[u_1, u_2,\dots ,u_n]({\bf E}_n^t-X)=-[E_{10}, E_{20},\dots ,E_{n0}]M_n(X)Y({\bf E}_n^t-X)$. By Lemma \ref{lem-y-com}, $Y$ commutes with ${\bf E}_n^t$ and with $X$ and thus with ${\bf E}_n^t-X$. Hence, 
$$M_n(X)Y({\bf E}_n^t-X)=M_n(X)({\bf E}_n^t-X)Y=C_n(X)Y=I_n$$ by Lemma \ref{matrix-inv-n}. Thus, $[u_1, u_2,\dots ,u_n]({\bf E}_n^t-X)=-[E_{10}, E_{20},\dots ,E_{n0}]$. Since $u_0=1$, this yields the desired result for $i=1,2,\dots ,n$. It remains to prove it for $i=0$.

We  have $[u_0, u_1, \dots, u_n]({\bf E}_{n+1}^t-X)=[q, 0, 0, \dots 0]$ for $q=\sum_{a=0}^n u_aE_{0a}-u_0X$, and need to verify that $q=0$. After multiplying by $M_{n+1}(X)$, we obtain 
$$0=[u_0, u_1, \dots, u_n]C_{n+1}(X)=[u_0, u_1, \dots, u_n]({\bf E}_{n+1}^t-X)M_{n+1}(X)=[q, 0, \dots, 0]M_{n+1}(X)$$ by Lemma \ref{matrix-inv-n} and using that $C_{n+1}(X)=0$. Thus, $q$ multiplied by any entry of the leftmost column of $M_{n+1}(X)$ equals $0$. In particular, $qM_{n+1}(X)_{00}=0$. This implies $q=0$ because $M_{n+1}(X)_{00}=C_n(X+1)$, because $C_n(X+1) \neq 0$ from Lemma \ref{unique-rep}, and because $U''$ is a domain. This completes the proof of the Lemma.
\end{proof}

\begin{lemma} \label{lemma-v-e-x} If $v_0,v_1,\dots ,v_n$ are such that $\sum_{a=0}^n v_aE_{ia}=v_iX$  for $i=0,1,\dots ,n$, then $v_i=v_0u_i$ for $i=0,1,\dots ,n$.
\end{lemma}
\begin{proof}
We need to show $[v_1,\dots ,v_n]  = v_0[u_1,\dots ,u_n]$. The identities given in the statement imply that $[v_1, \dots, v_n] (\mathbf{E}_n^t - X) = -v_0[E_{10}, E_{20},\dots ,E_{n0}]$. We multiply this matrix identity by $M_n(X)Y$ on the right. Then, using  the definition of $u_1,\dots ,u_n$, Lemma \ref{matrix-inv-n}, and the identity $C_n(X)Y=1$, we obtain $[v_1,\dots ,v_n]  = v_0[u_1,\dots ,u_n]$ as needed.  \end{proof}

\begin{lemma} \label{u-e-commutator} $[u_i,E_{jk}]=\delta_{k0}u_ju_i-\delta_{ki}u_j$ for $i,j,k$ from $0$ to $n$.
\end{lemma}
\begin{proof} Fix $j$ and $k$ and set $v_i=[u_i,E_{jk}]+\delta_{ki}u_j$ for $i=0,1,\dots ,n$. We will show that $v_0, \dots , v_n$ satisfy the hypothesis of Lemma \ref{lemma-v-e-x}.  Note first that  $v_0=[u_0,E_{jk}]+\delta_{k0}u_j=\delta_{k0}u_j$ as $u_0=1$. Now, for $i=0,1,\dots ,n$, by Lemma \ref{lemma-u-e-x} we have  that
\begin{eqnarray*}
\sum_{a=0}^n[u_a,E_{jk}]E_{ia}+\sum_{a=0}^n u_a(\delta_{aj}E_{ik}-\delta_{ki}E_{ja})& =& \sum_{a=0}^n ([u_a,E_{jk}]E_{ia}+u_a[E_{ia},E_{jk}])\\ & = & \sum_{a=0}^n [u_aE_{ia},E_{jk}]=[\sum_{a=0}^n u_aE_{ia},E_{jk}]\\ & = &[u_iX,E_{jk}]=[u_i,E_{jk}]X.
\end{eqnarray*}
On the other hand, 
$$ \sum_{a=0}^n u_a(\delta_{aj}E_{ik}-\delta_{ki}E_{ja})=\sum_{a=0}^n u_a\delta_{aj}E_{ik} - \sum_{a=0}^n u_a\delta_{ki}E_{ja} =u_jE_{ik}-\delta_{ki}u_jX$$
by Lemma \ref{lemma-u-e-x}. Thus $(\sum_{a=0}^n[u_a,E_{jk}]E_{ia})+u_jE_{ik}-\delta_{ki}u_jX=[u_i,E_{jk}]X$. Then 
$$\sum_{a=0}^n v_aE_{ia}=\left(\sum_{a=0}^n [u_a,E_{jk}]E_{ia}\right)+u_jE_{ik}=[u_i,E_{jk}]X+\delta_{ki}u_jX=v_iX.$$

Lemma \ref{lemma-v-e-x} implies that $v_i=v_0u_i=\delta_{k0}u_ju_i$. Thus $[u_i,E_{jk}]=v_i-\delta_{ki}u_j=\delta_{k0}u_ju_i-\delta_{ki}u_j,$ as claimed.
\end{proof}

\begin{lemma} \label{u-u-commute} $[u_i,u_j]=0$ for $0 \leq i,j \leq n$.
\end{lemma}
\begin{proof} Fix $j$ and let $v_i=[u_i,u_j]$. In particular, $v_0=0$.  Lemmas \ref{lemma-u-e-x} and \ref{u-e-commutator} imply  that  $\sum_{a=0}^n v_a E_{ia} = v_i(X+1)$ for all $i$. Since $v_0=0$, we can write the last set of identities in the following matrix form:  $[v_1, \dots ,v_n] (\mathbf{E}_n^t - (X+1)) = [0, \dots ,0]$. After  multiplying by $M_n(X+1)$ on the right and using Lemma \ref{matrix-inv-n}, we obtain $v_iC_n(X+1)=0$. Since $C_n(X+1) \neq 0$ (by Lemma \ref{unique-rep}) and $U''$ is a domain, $v_i=0$ for all $i=1, \dots ,n$. Thus we have the desired. 
\end{proof}

\begin{lemma} \label{lem-pi-g}
The correspondence $w \mapsto \pi_g(w)$, $w \in U$, $X \mapsto X - \frac{1}{n+1}G_1$ extends to a homomorphism $\pi_g': U' \to U'$ of associative unital algebras. Furthermore, $\ker \pi_g' = (G_1)$ in $U'$.
\end{lemma}
\begin{proof}
We can see that this correspondence yields a well-defined homomorphism $U(\mathfrak{g}_{n+1}^X) \to U(\mathfrak{g}_{n+1}^X)$. To see that it yields a well-defined map $U' \to U'$, note that the determinant definition of $C_{n+1}(T)$ easily implies the identity $\pi_g' (C_{n+1}(T)) = C_{n+1} \left(\pi_g'(T) + \frac{1}{n+1}G_1\right)$, so we will have $\pi_g'(C_{n+1}(X)) = C_{n+1}(X)$; thus we have a well-defined homomorphism $\pi_g': U' \to U'$. For the kernel of $\pi_g'$, we use  that $(\sum_{i=0}^n w_i X^i) - \pi_g'(\sum_{i=0}^n w_i X^i) \in (G_1)$ for every $\sum_{i=0}^n w_i X^i \in U'$.\end{proof}

Let $U_s'$ be the image of $\pi_g'$. By Lemma \ref{lem-pi-g}, $U_s' \simeq U'/(G_1)$. We have a natural embedding $\iota_s' : U_s' \to U'$ such that $\pi_g'\iota_s' = \mbox{Id}$. 

Recall $C_{\rho}(\mathcal E + R_1) = 0$ by Theorem \ref{rho-capelli}. Thus we may define a homomorphism  $\rho' : U' \to {\mathcal D}'(n) \otimes U(\mathfrak{gl}(n))$ of associative unital algebras by the identities $\rho'(w) = \rho (w)$ for all $w \in U(\mathfrak{gl}(n+1))$ and $\rho'(X) = \mathcal E + R_1$. Also, define $\rho_s' : U_s' \to {\mathcal D}'(n) \otimes U(\mathfrak{gl}(n))$ by $\rho_s' = \rho' \iota_s'$. One hence has the following diagram.

$$
\xymatrix{ U \ar@<0.5ex>[d]^{ \pi_g}   \ar[r]^{\iota}  &   U' \ar@<0.5ex>[d]^{ \pi_g'}   \ar[r]^{\rho'\hspace{1cm}}  &  \mathcal{D}'(n) \otimes U(\mathfrak{gl}(n))   \ar@<0.5ex>[d] ^{\sigma}\\
U_s \ar[u]^{\iota_s}    \ar[r]^{\iota} & U_s' \ar[u]^{\iota_s'}    \ar[r]^{\iota'}   \ar[ur]^{\rho_s'}  &  U''     }
$$
(We will define $\sigma$ momentarily.) Note again that $\rho$ and $\rho_s$ are obtained by taking the compositions of $\rho'$ and $\rho_s'$, respectively, with the natural embedding $\iota: U \to U'$.

\begin{proposition} \label{prop-sigma} Let  $S \in U''$ be such that $[S,u_i]=0$ for $i>0$,  $[S,E_{0i}]=0$ for $i>0$, and $[S, E_{ab} - \delta_{ab}E_{00}]=0$ for $a,b>0$. The correspondence 
\begin{eqnarray*}
\frac{t_i}{t_0}& \mapsto &u_i, \; \mbox{ for } i>0,\\
t_0\partial_j & \mapsto &E_{0j}-\delta_{0j}S, \; \mbox{ for } j\geq 0,\\
E_{ab} & \mapsto & E_{ab}-u_aE_{0b}-\delta_{ab}S, \; \mbox{ for all } a,b>0.
\end{eqnarray*}
extends to a  homomorphism $\sigma: \mathcal{D}'(n) \otimes U (\mathfrak{gl} (n)) \to U''$ of associative unital algebras. Furthermore, $\sigma \rho' = \iota'\pi_g'$ and $\sigma \rho = \iota' \iota \pi_g$.
\end{proposition}

\begin{remark}
We can see such $S$ exists and thus such $\sigma$ exists; e.g., we could take $S=0$.
\end{remark}

\begin{proof} Note that  $[S,E_{i0} + u_iE_{00}] = 0$ by Lemma \ref{lemma-u-e-x}. To check that the correspondence extends to a homomorphism, we verify that $\sigma([x,y]) = [\sigma(x),\sigma(y)]$ whenever $x,y$ equal one of the generators $\frac{t_i}{t_0}, t_0\partial_i, E_{ab}$.

By Lemma \ref{u-u-commute}, we have $\sigma([\frac{t_i}{t_0},\frac{t_j}{t_0}])=\sigma(0)=0=[u_i,u_j]=[\sigma(\frac{t_i}{t_0}),\sigma(\frac{t_j}{t_0})]$, as desired.

By Lemma \ref{u-e-commutator}, we have $\sigma([\frac{t_i}{t_0},t_0\partial_j])=\sigma(\delta_{0j}\frac{t_i}{t_0}-\delta_{ij})=\delta_{0j}u_i-\delta_{ij}=[u_i,E_{0j}]=[\sigma(\frac{t_i}{t_0}),\sigma(t_0\partial_j)]$ as $u_0=1$ and $[u_i,S]=0$.

By Lemma \ref{u-e-commutator} and Lemma \ref{u-u-commute}, $\sigma([\frac{t_i}{t_0},E_{ab}])=\sigma(0)=0=\delta_{b0}u_au_i-\delta_{bi}u_a-u_a(\delta_{b0}u_i-\delta_{bi}) =[u_i,E_{ab}]-[u_i,u_a]E_{0b}-u_a[u_i,E_{0b}]=[u_i,E_{ab}-u_aE_{0b}-\delta_{ab}S]=[\sigma(\frac{t_i}{t_0}),\sigma(E_{ab})]$, as $u_0=1$ and $[u_i,u_a]=0$ and $[u_i,S]=0$.

We have $\sigma([t_0\partial_i,t_0\partial_j])=\sigma(\delta_{i0}t_0\partial_j-\delta_{j0}t_0\partial_i)=\delta_{i0}(E_{0j}-\delta_{j0}S)-\delta_{j0}(E_{0i}-\delta_{i0}S)=\delta_{i0}E_{0j}-\delta_{j0}E_{0i}=[E_{0i},E_{0j}]=[\sigma(t_0\partial_i),\sigma(t_0\partial_j)]$, by the definition of $S$.

By Lemma \ref{u-e-commutator} we have $\sigma([t_0\partial_i,E_{ab}]) = 0 = [E_{0i} - \delta_{0i}S,E_{ab}- u_aE_{0b}- \delta_{ab}S] = [\sigma(t_0\partial_i),\sigma(E_{ab})])$

Finally, using again Lemma \ref{u-e-commutator} we have
\begin{eqnarray*}
\sigma([E_{ab}, E_{cd}]) & = &\sigma(\delta_{bc} E_{ad}- \delta_{ad} E_{cb}) = \delta_{bc} (E_{ad} - u_a E_{0d} - \delta_{ad}S) - \delta_{ad} (E_{cb} - u_c E_{0b} - \delta_{cb}S)\\
& = & \delta_{bc} E_{ad} - \delta_{ad} E_{cb} -\delta_{bc} u_a E_{0d}  +  \delta_{ad}  u_c E_{0b} \\
& = &  [E_{ab}, E_{cd}] + [E_{ab},-u_cE_{0d}] + [-u_aE_{0b},E_{cd}] + [- u_a E_{0b}, - u_c E_{0d}] \\
& = & [E_{ab} - u_a E_{0b}, E_{cd} - u_c E_{0d}] + [E_{ab}, - \delta_{cd}S] + [-\delta_{ab}S,E_{cd}]\\
& = &  [E_{ab} - u_a E_{0b} - \delta_{ab}S, E_{cd} - u_c E_{0d} - \delta_{cd}S] \\
& = & [\sigma(E_{ab}), \sigma(E_{cd})]. 
\end{eqnarray*}
In the above sequence of identities  we used that $[S,\delta_{cd} E_{ab} - \delta_{ab}E_{cd}] = [S, \delta_{cd}(E_{ab}-\delta_{ab}E_{00})-\delta_{ab}(E_{cd}-\delta_{cd}E_{00})] = 0$ and that $[E_{ab},-u_cE_{0d}] + [-u_aE_{0b},E_{cd}] + [- u_a E_{0b}, - u_c E_{0d}] = -\delta_{bc}u_a E_{0d} + \delta_{ad}u_cE_{0b}$. 

Thus $\sigma$ is indeed a homomorphism of associative unital algebras.

Note that $\sigma(\mathcal{E}) = X-S$ and $\sigma\left(R_1\right) = S - \frac{1}{n+1}G_1$. Using the definitions of $\sigma$ and $\rho'$, it is easy to verify that $\sigma \rho' (X) = X- \frac{1}{n+1}G_1$, $\sigma \rho' (E_{ij}) = E_{ij}$ for $i \neq j$, and $\sigma \rho' (E_{ii}) = E_{ii} - \frac{1}{n+1}G_1$. Hence, $\sigma \rho' = \iota'\pi_g'$. The identity $\sigma \rho = \iota' \iota \pi_g$ follows from $\rho = \rho'\iota$ and $\pi_g'\iota = \iota \pi_g$.\end{proof}

\begin{theorem} We have the following:
\begin{itemize}
\item[(i)] $\ker \rho' = (G_1)$ in $U'$, and $\ker \rho = (G_1)$ in $U$.
\item[(ii)] $\ker \rho_s' = (0)$ in $U_s'$, and $\ker \rho_s = (0)$ in $U_s$.
\end{itemize}
\end{theorem}

\begin{proof} For part (i) we first note that $\rho'(G_1) = \rho(G_1) = 0$. To complete the proof we use Lemma \ref{lem-pi-g} along with Proposition \ref{prop-sigma} and the fact that the kernel of $\pi_g$ is $(G_1)$ in $U$. As for part (ii), note that $\pi_s'\pi_s'=\pi_s'$; then, if $t \in \ker \rho_s' = (G_1) \cap U_s$, then we have $t=\pi_s'(t)=0$. Thus $\ker \rho_s' = \{0\}$. Then $\ker \rho_s \subset \ker \rho_s'$, so $\ker \rho_s = \{0\}$.  \end{proof}

\newpage
\section*{Appendix. Formulas for the images of certain elements under $\rho$} 
In this appendix we provide explicit formulas for the images under $\rho$ of $r_k^{\mathfrak{gl}(n+1)}(a,b)$. For the definition of the latter see  \eqref{def-r-k}. The proofs of these formulas are independent of the results in the previous sections of the paper.  In this way we have an alternative proof of Theorem \ref{thm-rho-g}. It is interesting to note that one can then go ``backwards'' and prove Theorem \ref{rho-capelli} and then Theorem \ref{thm-hc-rho} purely computationally. Thus, this appendix leads to an alternative (more computational) approach of the results established in Sections \ref{sec-5} and \ref{sec-6}.

Note that $r_{k+1}^{\mathfrak{gl}(N)}(a,b) = \sum_{i}r_k^{\mathfrak{gl}(N)}(a,i)E_{ib}$ for all nonnegative integers $k$, for $N=n$ or $n+1$, and for all $a,b$.

For a positive integer $m$ and for $a$ and $b$ with $0\leq a \leq n$ and $1\leq b \leq n$, define $$f_m(a,b)=\sum_{i=1}^n t_a \partial_i \otimes r_{m-1}^{\mathfrak{gl}(n)}(i,b).$$

\medskip
\noindent {\bf Theorem.} \emph{ For all positive integers $k$, $a$, $b$, such that $a,b\leq n$,}
\begin{eqnarray*}
\rho(r_k^{\mathfrak{gl}(n+1)}(a,b))& = & \sum_{m=1}^k \left( f_m(a,b) \sum_{g=0}^{k-m} \binom{k}{g} R_1^g R_2^{k-m-g} \right)+\sum_{g=0}^k \binom{k}{g} R_1^g \left( 1 \otimes r_{k-g}^{\mathfrak{gl}(n)}(a,b) \right), \\
\rho(r_k^{\mathfrak{gl}(n+1)}(a,0))& = &\left( \sum_{g=0}^{k-1} \binom{k}{g} R_1^g R_2^{k-1-g} \right)\left( t_a \partial_0 \otimes 1 \right) - \sum_{g=0}^{k-1} \left( \binom{k}{g} R_1^g \sum_{j>0} \frac{t_j}{t_0} \otimes r_{k-g}^{\mathfrak{gl}(n)}(a,j) \right) \\
&& - \left( \frac{t_a}{t_0} \otimes 1 \right) \sum_{m=2}^k \left( \sum_{g=0}^{k-m} \binom{k}{g} R_1^g R_2^{k-m-g} \right)\left( \sum_{i,j>0} \partial_i t_j \otimes r_{m-1}^{\mathfrak{gl}(n)}(i,j) \right),\\
\rho(r_k^{\mathfrak{\mathfrak{gl}}(n+1)}(0,b))& = & \sum_{m=1}^k \left( f_m(0,b) \sum_{g=0}^{k-m} \binom{k}{g} R_1^g R_2^{k-m-g} \right),\\
\rho(r_k^{\mathfrak{\mathfrak{gl}}(n+1)}(0,0))& = & \left( \sum_{g=0}^{k-1} \binom{k}{g} R_1^g R_2^{k-1-g} \right)\left( t_0 \partial_0 \otimes 1 \right)+R_1^k\\
&& - \sum_{m=2}^k \left( \sum_{g=0}^{k-m} \binom{k}{g} R_1^g R_2^{k-m-g} \right)\left( \sum_{i,j>0} \partial_i t_j \otimes r_{m-1}^{\mathfrak{gl}(n)}(i,j) \right).
\end{eqnarray*}

\medskip

\begin{proof} We prove all four statements simultaneously by induction on $k$. The base case $k=1$ follows from the definition of $\rho$.  Suppose the formulas in the statement of the Theorem are true for some positive integer $k$. Let us prove them for $k+1$.

\bigskip
First, we consider the value of $\rho(r_k^{\mathfrak{gl}(n+1)}(a,b))$ for $a,b>0$. We have

\begin{small}
\begin{eqnarray*}
\rho(r_{k+1}^{\mathfrak{gl}(n+1)}(a,b)) &= & \rho(r_k^{\mathfrak{gl}(n+1)}(a,0)E_{0b}+\sum_{i>0}r_k^{\mathfrak{gl}(n+1)}(a,i)E_{ib}) \\
& = & \rho(r_k^{\mathfrak{gl}(n+1)}(a,0))\rho(E_{0b})+\sum_{i>0}\rho(r_k^{\mathfrak{gl}(n+1)}(a,i))\rho(E_{ib}) \\
& =  &    \Biggl( \left( \sum_{g=0}^{k-1} \binom{k}{g} R_1^g R_2^{k-1-g} \right)\left( t_a \partial_0 \otimes 1 \right) - \sum_{g=0}^{k-1} \left( \binom{k}{g} R_1^g \sum_{j>0} \frac{t_j}{t_0} \otimes r_{k-g}^{\mathfrak{gl}(n)}(a,j) \right)\\
& - & \left( \frac{t_a}{t_0} \otimes 1 \right) \sum_{m=2}^k \left( \sum_{g=0}^{k-m} \binom{k}{g} R_1^g R_2^{k-m-g} \right)\left( \sum_{i,j>0} \partial_i t_j \otimes r_{m-1}^{\mathfrak{gl}(n)}(i,j) \right) \Biggr)(t_0 \partial_b \otimes 1) \\
& + & \sum_{i>0} \Biggl( \sum_{m=1}^k \left( f_m(a,i) \sum_{g=0}^{k-m} \binom{k}{g} R_1^g R_2^{k-m-g} \right) \\
& + & \sum_{g=0}^k \binom{k}{g} R_1^g \left( 1 \otimes r_{k-g}^{\mathfrak{gl}(n)}(a,i) \right) \Biggr)(t_i \partial_b \otimes 1 + 1 \otimes E_{ib} + \delta_{ib}R_1) \\
& = &  \left( \sum_{g=0}^{k-1} \binom{k}{g} R_1^g R_2^{k-1-g} \right)\left( t_a \partial_0 t_0 \partial_b \otimes 1 \right) - \sum_{g=0}^{k-1} \left( \binom{k}{g} R_1^g \sum_{j>0} t_j \partial_b \otimes r_{k-g}^{\mathfrak{gl}(n)}(a,j) \right)\\
& - & \left( \frac{t_a}{t_0} \otimes 1 \right) \sum_{m=2}^k \left( \sum_{g=0}^{k-m} \binom{k}{g} R_1^g R_2^{k-m-g} \right)\left( \sum_{i,j>0} \partial_i t_j \otimes r_{m-1}^{\mathfrak{gl}(n)}(i,j) \right)(t_0 \partial_b \otimes 1) \\
& + & \sum_{i>0} \sum_{m=1}^k \left( f_m(a,i) (t_i \partial_b \otimes 1)  \sum_{g=0}^{k-m} \binom{k}{g} R_1^g R_2^{k-m-g} \right) +  \sum_{i>0} \sum_{g=0}^k \binom{k}{g} R_1^g \left( t_i \partial_b \otimes r_{k-g}^{\mathfrak{gl}(n)}(a,i) \right) \\
& + &  \sum_{i>0} \sum_{m=1}^k \left( f_m(a,i) (1 \otimes E_{ib})  \sum_{g=0}^{k-m} \binom{k}{g} R_1^g R_2^{k-m-g} \right) + \sum_{i>0} \sum_{g=0}^k \binom{k}{g} R_1^g \left( 1 \otimes r_{k-g}^{\mathfrak{gl}(n)}(a,i) E_{ib} \right) \\
& + & \sum_{m=1}^k \left( f_m(a,b) \sum_{g=0}^{k-m} \binom{k}{g} R_1^{g+1} R_2^{k-m-g} \right) + \sum_{g=0}^k \binom{k}{g} R_1^{g+1} \left( 1 \otimes r_{k-g}^{\mathfrak{gl}(n)}(a,b) \right).
\end{eqnarray*}
\end{small}

Write the last expression in the form $X_1-X_2-X_3+X_4+X_5+X_6+X_7+X_8+X_9$. Then $X_5-X_2=\sum_{i>0} \sum_{g=0}^k \binom{k}{g} R_1^g \left( t_i \partial_b \otimes r_{k-g}^{\mathfrak{gl}(n)}(a,i) \right) - \sum_{g=0}^{k-1} \left( \binom{k}{g} R_1^g \sum_{j>0} t_j \partial_b \otimes r_{k-g}^{\mathfrak{gl}(n)}(a,j) \right)= R_1^k(t_a\partial_b \otimes 1).$

On the other hand, $X_4=\sum_{m=1}^k((\sum_{j,i>0} t_a\partial_j t_i \partial_b \otimes r_{m-1}^{\mathfrak{gl}(n)}(j,i))(\sum_{g=0}^{k-m} \binom{k}{g} R_1^gR_2^{k-m-g})).$ Also, $X_3= \sum_{m=2}^k((\sum_{j,i>0} t_a\partial_j t_i \partial_b \otimes r_{m-1}^{\mathfrak{gl}(n)}(j,i))(\sum_{g=0}^{k-m} \binom{k}{g} R_1^gR_2^{k-m-g})).$ Then $X_4-X_3=(\sum_{j,i>0} t_a\partial_jt_i\partial_b \otimes r_0^{\mathfrak{gl}(n)}(j,i))(\sum_{g=0}^{k-1} \binom{k}{g}R_1^gR_2^{k-1-g})=(\sum_{i>0}t_a\partial_it_i\partial_b \otimes 1)(\sum_{g=0}^{k-1} \binom{k}{g}R_1^gR_2^{k-1-g}),$ so $X_1+X_4-X_3=(\sum_{i \geq 0}t_a\partial_it_i\partial_b \otimes 1)(\sum_{g=0}^{k-1} \binom{k}{g}R_1^gR_2^{k-1-g}) = (t_a\partial_b \otimes 1)(\sum_{g=0}^{k-1} \binom{k}{g}R_1^gR_2^{k-g}),$ as $\sum_{i\geq 0} t_a \partial_i t_i \partial_b \otimes 1 = (t_a \partial_b \otimes 1)R_2.$

Therefore we have $(X_1+X_4-X_3)+(X_5-X_2) = (t_a\partial_b \otimes 1)(\sum_{g=0}^k \binom{k}{g}R_1^gR_2^{k-g}).$

Furthermore, since $\sum_{i>0} f_m(a,i)(1 \otimes E_{ib}) = f_{m+1}(a,b)$, we have\\ $X_6=\sum_{m=1}^k(f_{m+1}(a,b)\sum_{g=0}^{k-m}\binom{k}{g}R_1^gR_2^{k-m-g})=\sum_{m=2}^{k+1}f_m(a,b)(\sum_{g=0}^{k+1-m}\binom{k}{g}R_1^gR_2^{k+1-m-g})$. \\ Thus $X_6+((X_1+X_4-X_3)+(X_5-X_2)) = \sum_{m=1}^{k+1}f_m(a,b)(\sum_{g=0}^{k+1-m}\binom{k}{g}R_1^gR_2^{k+1-m-g}).$ \\ But  $X_8=\sum_{m=1}^k ( f_m(a,b) \sum_{g=1}^{k-m+1} \binom{k}{g-1} R_1^g R_2^{k-m-g+1}).$ Hence,\\ $(X_6+((X_1+X_4-X_3)+(X_5-X_2)))+X_8 = \sum_{m=1}^{k+1}(f_m(a,b)\sum_{g=0}^{k+1-m}\binom{k+1}{g}R_1^gR_2^{k+1-m-g}).$

Now using that $\sum_{i>0} r_{k-g}^{\mathfrak{gl}(n)}(a,i)E_{ib} = r_{k+1-g}^{\mathfrak{gl}(n)}(a,b)$, we find $X_7=\sum_{g=0}^k\binom{k}{g}R_1^g(1\otimes r_{k+1-g}^{\mathfrak{gl}(n)}(a,b))$. Also $X_9=\sum_{g=1}^{k+1} \binom{k}{g-1} R_1^g (1\otimes r_{k+1-g}^{\mathfrak{gl}(n)}(a,b))$ implies that $X_7+X_9=\sum_{g=0}^{k+1} \binom{k+1}{g}R_1^g (1\otimes r_{k+1-g}^{\mathfrak{gl}(n)}(a,b)).$

Therefore, $\rho(r_{k+1}^{\mathfrak{gl}(n+1)}(a,b)) = ((X_6+((X_1+X_4-X_3)+(X_5-X_2)))+X_8) + (X_7+X_9) = \\ \sum_{m=1}^{k+1}(f_m(a,b)\sum_{g=0}^{k+1-m}\binom{k+1}{g}R_1^gR_2^{k+1-m-g} + \sum_{g=0}^{k+1} \binom{k+1}{g}R_1^g (1\otimes r_{k+1-g}^{\mathfrak{gl}(n)}(a,b)).$ This completes the proof of the inductive step for $\rho(r_{k+1}^{\mathfrak{gl}(n+1)}(a,b))$ for $a,b>0$.

\bigskip

Next, we consider the value of $\rho(r_k^{\mathfrak{gl}(n+1)}(a,0))$ for $a>0$. We have

\begin{eqnarray*}
\rho(r_{k+1}^{\mathfrak{gl}(n+1)}(a,0)) & = & \rho(r_k^{\mathfrak{gl}(n+1)}(a,0)E_{00} + \sum_{i>0} r_k^{\mathfrak{gl}(n+1)}(a,i)E_{i0}) \\
& = & \rho(r_k^{\mathfrak{gl}(n+1)}(a,0))\rho(E_{00}) + \sum_{i>0} \rho(r_k^{\mathfrak{gl}(n+1)}(a,i))\rho(E_{i0}) \\
& = & \Biggl( \left( \sum_{g=0}^{k-1} \binom{k}{g} R_1^g R_2^{k-1-g} \right)\left( t_a \partial_0 \otimes 1 \right) - \sum_{g=0}^{k-1} \left( \binom{k}{g} R_1^g \sum_{j>0} \frac{t_j}{t_0} \otimes r_{k-g}^{\mathfrak{gl}(n)}(a,j) \right) \\
& - & \left( \frac{t_a}{t_0} \otimes 1 \right) \sum_{m=2}^k \left( \sum_{g=0}^{k-m} \binom{k}{g} R_1^g R_2^{k-m-g} \right)\left( \sum_{i,j>0} \partial_i t_j \otimes r_{m-1}^{\mathfrak{gl}(n)}(i,j) \right) \Biggr) (t_0 \partial_0 \otimes 1 + R_1) \\
& + & \sum_{i>0}\Biggl( \sum_{m=1}^k \left( f_m(a,i) \sum_{g=0}^{k-m} \binom{k}{g} R_1^g R_2^{k-m-g} \right)\\ 
& + & \sum_{g=0}^k \binom{k}{g} R_1^g \left( 1 \otimes r_{k-g}^{\mathfrak{gl}(n)}(a,i) \right) \Biggr)(t_i\partial_0 \otimes 1 - \sum_{j>0} \frac{t_j}{t_0} \otimes E_{ij}) \\
& = & \left(\sum_{g=0}^{k-1} \binom{k}{g} R_1^g R_2^{k-1-g} \right)(t_a \partial_0 t_0 \partial_0 \otimes 1) - \sum_{g=0}^{k-1} \left( \binom{k}{g} R_1^g \sum_{j>0} t_j\partial_0 \otimes r_{k-g}^{\mathfrak{gl}(n)}(a,j) \right) \\
& - & \sum_{m=2}^k \left( \sum_{g=0}^{k-m} \binom{k}{g} R_1^g R_2^{k-m-g} \right)\left( \sum_{i,j>0} t_a \partial_i t_j \partial_0 \otimes r_{m-1}^{\mathfrak{gl}(n)}(i,j) \right) \\
& + & \left( \sum_{g=0}^{k-1} \binom{k}{g} R_1^{g+1} R_2^{k-1-g} \right)\left( t_a \partial_0 \otimes 1 \right) - \sum_{g=0}^{k-1} \left( \binom{k}{g} R_1^{g+1} \sum_{j>0} \frac{t_j}{t_0} \otimes r_{k-g}^{\mathfrak{gl}(n)}(a,j) \right) \\
& - & \left( \frac{t_a}{t_0} \otimes 1 \right) \sum_{m=2}^k \left( \sum_{g=0}^{k-m} \binom{k}{g} R_1^{g+1} R_2^{k-m-g} \right)\left( \sum_{i,j>0} \partial_i t_j \otimes r_{m-1}^{\mathfrak{gl}(n)}(i,j) \right) 
\end{eqnarray*}
\begin{eqnarray*}
& + & \sum_{m=1}^k \left(\sum_{i>0} f_m(a,i)(t_i \partial_0 \otimes 1) \right) \left(\sum_{g=0}^{k-m} \binom{k}{g} R_1^g R_2^{k-m-g}\right) \\
& + & \sum_{g=0}^k \binom{k}{g} R_1^g \left(\sum_{i>0} t_i \partial_0 \otimes r_{k-g}^{\mathfrak{gl}(n)}(a,i)\right) \\
& - & \sum_{m=1}^k \left(\sum_{i,j>0} f_m(a,i)\left(\frac{t_j}{t_0} \otimes E_{ij}\right) \right) \left(\sum_{g=0}^{k-m} \binom{k}{g} R_1^g R_2^{k-m-g}\right) \\
& - & \sum_{g=0}^k \binom{k}{g} R_1^g \left(\sum_{i,j>0} \frac{t_j}{t_0} \otimes r_{k-g}^{\mathfrak{gl}(n)}(a,i)E_{ij}\right).
\end{eqnarray*}

We write the last expression as  $X_1 - X_2 - X_3 + X_4 - X_5 - X_6 + X_7 + X_8 - X_9 - X_{10}$. Then $X_8-X_2 = \sum_{g=0}^k \binom{k}{g} R_1^g \left(\sum_{i>0} t_i \partial_0 \otimes r_{k-g}^{\mathfrak{gl}(n)}(a,i)\right) - \sum_{g=0}^{k-1} \binom{k}{g} R_1^g \left(\sum_{i>0} t_i \partial_0 \otimes r_{k-g}^{\mathfrak{gl}(n)}(a,i)\right) = R_1^k(t_a \partial_0 \otimes 1)$.

Also, $X_5+X_{10} = \left(\sum_{g=1}^k \binom{k}{g-1} R_1^g \sum_{j>0} \frac{t_j}{t_0} \otimes r_{k+1-g}^{\mathfrak{gl}(n)}(a,j)\right) + \left(\sum_{g=0}^k \binom{k}{g} R_1^g \sum_{j>0} \frac{t_j}{t_0} \otimes r_{k+1-g}^{\mathfrak{gl}(n)}(a,j)\right) \\ = \left(\sum_{g=0}^k \binom{k+1}{g} R_1^g \sum_{j>0} \frac{t_j}{t_0} \otimes r_{k+1-g}^{\mathfrak{gl}(n)}(a,j)\right)$.

We next have $X_9 = \sum_{m=1}^k \Biggl( \sum_{i,j>0} f_m(a,i) \left(\frac{t_j}{t_0} \otimes E_{ij}\right)\Biggr)\left(\sum_{g=0}^{k-m} \binom{k}{g} R_1^gR_2^{k-m-g}\right)$. Now, for any positive integer $m$, $\sum_{i,j>0} f_m(a,i) \left(\frac{t_j}{t_0} \otimes E_{ij}\right) = \sum_{i,j>0} \left(\sum_{\ell>0} t_a \partial_{\ell} \otimes r_{m-1}^{\mathfrak{gl}(n)}(\ell,i)\right) \left(\frac{t_j}{t_0} \otimes E_{ij}\right) = \sum_{j,\ell>0} t_a\partial_{\ell}t_jt_0^{-1} \otimes \left(\sum_{i>0} r_{m-1}^{\mathfrak{gl}(n)}(\ell,i)E_{ij}\right) = \sum_{j,\ell>0} t_a\partial_{\ell}t_jt_0^{-1} \otimes r_m^{\mathfrak{gl}(n)}(\ell,j) = \\ \left(\frac{t_a}{t_0} \otimes 1\right)\sum_{j,\ell>0} \partial_{\ell}t_j \otimes r_m^{\mathfrak{gl}(n)}(\ell,j)$, which is the same as $\left(\frac{t_a}{t_0} \otimes 1\right)\sum_{i,j>0} \partial_it_j \otimes r_m^{\mathfrak{gl}(n)}(i,j)$.  Then \\ $X_9 =  \left(\frac{t_a}{t_0} \otimes 1\right)\sum_{m=2}^{k+1} \left(\sum_{g=0}^{k+1-m} \binom{k}{g} R_1^gR_2^{k+1-m-g}\right) \left(\sum_{i,j>0} \partial_it_j \otimes r_{m-1}^{\mathfrak{gl}(n)}(i,j)\right)$. Also,\\ $X_6= \left(\frac{t_a}{t_0} \otimes 1\right)\sum_{m=2}^k \left( \sum_{g=1}^{k+1-m} \binom{k}{g-1} R_1^{g} R_2^{k+1-m-g} \right)\left( \sum_{i,j>0} \partial_i t_j \otimes r_{m-1}^{\mathfrak{gl}(n)}(i,j) \right)$. Therefore, \\ $X_6+X_9= \left(\frac{t_a}{t_0} \otimes 1\right)\sum_{m=2}^{k+1} \left(\sum_{g=0}^{k+1-m} \binom{k+1}{g} R_1^gR_2^{k+1-m-g}\right) \left(\sum_{i,j>0} \partial_it_j \otimes r_{m-1}^{\mathfrak{gl}(n)}(i,j)\right)$.

Next, for any positive integer $m$, $\sum_{i>0} f_m(a,i)(t_i\partial_0 \otimes 1) = \sum_{i,j>0} t_a\partial_jt_i\partial_0 \otimes r_{m-1}^{\mathfrak{gl}(n)}(j,i) = \sum_{i,j>0} t_a\partial_it_j\partial_0 \otimes r_{m-1}^{\mathfrak{gl}(n)}(i,j)$. Then $X_7 = \sum_{m=1}^k \left(\sum_{i,j>0} t_a\partial_it_j\partial_0 \otimes r_{m-1}^{\mathfrak{gl}(n)}(i,j) \right) \left(\sum_{g=0}^{k-m} \binom{k}{g} R_1^g R_2^{k-m-g}\right)$. On the other hand, $X_3 = \sum_{m=2}^k \left(\sum_{i,j>0} t_a\partial_it_j\partial_0 \otimes r_{m-1}^{\mathfrak{gl}(n)}(i,j) \right) \left(\sum_{g=0}^{k-m} \binom{k}{g} R_1^g R_2^{k-m-g}\right)$. Thus $X_7-X_3 = (\sum_{i,j>0} t_a\partial_it_j\partial_0 \otimes \delta_{ij}) \left(\sum_{g=0}^{k-1} \binom{k}{g} R_1^g R_2^{k-1-g}\right)=\left(\sum_{g=0}^{k-1} \binom{k}{g} R_1^g R_2^{k-1-g}\right)(t_a\partial_0 \otimes 1)(R_2-(t_0\partial_0 \otimes 1))$, and subsequently $X_7-X_3+X_1 = \left(\sum_{g=0}^{k-1} \binom{k}{g} R_1^g R_2^{k-1-g}\right)(t_a\partial_0 \otimes 1)R_2 =\\ \left(\sum_{g=0}^{k-1} \binom{k}{g} R_1^g R_2^{k-g}\right)(t_a\partial_0 \otimes 1)$. Therefore, $X_7-X_3+X_1+X_8-X_2 = \left(\sum_{g=0}^{k} \binom{k}{g} R_1^g R_2^{k-g}\right)(t_a\partial_0 \otimes 1)$. Since $X_4 = \left(\sum_{g=1}^{k} \binom{k}{g-1} R_1^g R_2^{k-g}\right)(t_a\partial_0 \otimes 1)$, \\ $X_7-X_3+X_1+X_8-X_2+X_4 = \left(\sum_{g=0}^{k} \binom{k+1}{g} R_1^g R_2^{k-g}\right)(t_a\partial_0 \otimes 1)$.

Hence, $\rho(r_{k+1}^{\mathfrak{gl}(n+1)}(a,0)) = (X_7-X_3+X_1+X_8-X_2+X_4) - (X_5+X_{10}) - (X_6 + X_9) = \left(\sum_{g=0}^{k} \binom{k+1}{g} R_1^g R_2^{k-g}\right)(t_a\partial_0 \otimes 1) - \left(\sum_{g=0}^k \binom{k+1}{g} R_1^g \sum_{j>0} \frac{t_j}{t_0} \otimes r_{k+1-g}^{\mathfrak{gl}(n)}(a,j)\right)\\ - \left(\frac{t_a}{t_0} \otimes 1\right)\sum_{m=2}^{k+1} \left(\sum_{g=0}^{k+1-m} \binom{k+1}{g} R_1^gR_2^{k+1-m-g}\right) \left(\sum_{i,j>0} \partial_it_j \otimes r_{m-1}^{\mathfrak{gl}(n)}(i,j)\right)$, as desired. This completes the inductive step for the value of $\rho(r_{k+1}^{\mathfrak{gl}(n+1)}(a,0))$.

\bigskip

Next, we consider the value of $\rho(r_k^{\mathfrak{gl}(n+1)}(0,b))$ for $b>0$.

\begin{eqnarray*}
\rho(r_{k+1}^{\mathfrak{gl}(n+1)}(0,b)) &= & \rho(r_k^{\mathfrak{gl}(n+1)}(0,0)E_{0b}+\sum_{i>0}r_k^{\mathfrak{gl}(n+1)}(0,i)E_{ib}) \\
& = & \rho(r_k^{\mathfrak{gl}(n+1)}(0,0))\rho(E_{0b})+\sum_{i>0}\rho(r_k^{\mathfrak{gl}(n+1)}(0,i))\rho(E_{ib}) \\
& =  &    \Biggl( \left( \sum_{g=0}^{k-1} \binom{k}{g} R_1^g R_2^{k-1-g} \right)\left( t_0 \partial_0 \otimes 1 \right) + R_1^k\\
& - &  \sum_{m=2}^k \left( \sum_{g=0}^{k-m} \binom{k}{g} R_1^g R_2^{k-m-g} \right)\left( \sum_{i,j>0} \partial_i t_j \otimes r_{m-1}^{\mathfrak{gl}(n)}(i,j) \right) \Biggr)(t_0 \partial_b \otimes 1) \\
& + & \sum_{i>0} \Biggl( \sum_{m=1}^k \left( f_m(0,i) \sum_{g=0}^{k-m} \binom{k}{g} R_1^g R_2^{k-m-g} \right) \Biggr)(t_i \partial_b \otimes 1 + 1 \otimes E_{ib} + \delta_{ib}R_1) \\
& = &  \left( \sum_{g=0}^{k-1} \binom{k}{g} R_1^g R_2^{k-1-g} \right)\left( t_0 \partial_0 t_0 \partial_b \otimes 1 \right)  + R_1^k(t_0\partial_b \otimes 1)\\
& - & \sum_{m=2}^k \left( \sum_{g=0}^{k-m} \binom{k}{g} R_1^g R_2^{k-m-g} \right)\left( \sum_{i,j>0} \partial_i t_j \otimes r_{m-1}^{\mathfrak{gl}(n)}(i,j) \right)(t_0 \partial_b \otimes 1) \\
& + & \sum_{i>0} \sum_{m=1}^k \left( f_m(0,i) (t_i \partial_b \otimes 1)  \sum_{g=0}^{k-m} \binom{k}{g} R_1^g R_2^{k-m-g} \right) \\
& + &  \sum_{i>0} \sum_{m=1}^k \left( f_m(0,i) (1 \otimes E_{ib})  \sum_{g=0}^{k-m} \binom{k}{g} R_1^g R_2^{k-m-g} \right) \\
& + & \sum_{m=1}^k \left( f_m(0,b) \sum_{g=0}^{k-m} \binom{k}{g} R_1^{g+1} R_2^{k-m-g} \right).
\end{eqnarray*}

We write the last  expression as $X_1+X_2-X_3+X_4+X_5+X_6$.

Now $X_4=\sum_{m=1}^k((\sum_{j,i>0} t_0\partial_j t_i \partial_b \otimes r_{m-1}^{\mathfrak{gl}(n)}(j,i))(\sum_{g=0}^{k-m} \binom{k}{g} R_1^gR_2^{k-m-g})).$ Also,\\ $X_3=  \sum_{m=2}^k((\sum_{j,i>0} t_0\partial_j t_i \partial_b \otimes r_{m-1}^{\mathfrak{gl}(n)}(j,i))(\sum_{g=0}^{k-m} \binom{k}{g} R_1^gR_2^{k-m-g})).$ Then \\ $X_4-X_3=(\sum_{j,i>0} t_a\partial_jt_i\partial_b \otimes r_0^{\mathfrak{gl}(n)}(j,i))(\sum_{g=0}^{k-1} \binom{k}{g}R_1^gR_2^{k-1-g})=(\sum_{i>0}t_0\partial_it_i\partial_b \otimes 1)(\sum_{g=0}^{k-1} \binom{k}{g}R_1^gR_2^{k-1-g}),$\\ so $X_1+X_4-X_3=(\sum_{i \geq 0}t_0\partial_it_i\partial_b \otimes 1)(\sum_{g=0}^{k-1} \binom{k}{g}R_1^gR_2^{k-1-g}) = (t_0\partial_b \otimes 1)(\sum_{g=0}^{k-1} \binom{k}{g}R_1^gR_2^{k-g}),$ as $\sum_{i\geq 0} t_0 \partial_i t_i \partial_b \otimes 1 = (t_0 \partial_b \otimes 1)R_2.$ Then $X_2+X_1+X_4-X_3=(t_0\partial_b \otimes 1)(\sum_{g=0}^k \binom{k}{g}R_1^gR_2^{k-g}).$

Since $\sum_{i>0} f_m(0,i)(1 \otimes E_{ib}) = f_{m+1}(0,b)$,  $X_5=\sum_{m=1}^k (f_{m+1}(0,b) \sum_{g=0}^{k-m} \binom{k}{g} R_1^g R_2^{k-m-g})=\sum_{m=2}^{k+1} (f_m(0,b) \sum_{g=0}^{k+1-m} \binom{k}{g} R_1^g R_2^{k+1-m-g})$. Then $X_5+(X_2+X_1+X_4-X_3)= \\ \sum_{m=1}^{k+1} (f_m(0,b) \sum_{g=0}^{k+1-m} \binom{k}{g} R_1^g R_2^{k+1-m-g})$. We also have  $X_6=\sum_{m=1}^k ( f_m(0,b) \sum_{g=1}^{k+1-m} \binom{k}{g-1} R_1^g R_2^{k+1-m-g} ).$
Therefore, $\rho(r_{k+1}^{\mathfrak{gl}(n+1)}(0,b))=(X_5+(X_2+X_1+X_4-X_3))+X_6=\sum_{m=1}^k ( f_m(0,b) \sum_{g=1}^{k+1-m} \binom{k+1}{g} R_1^g R_2^{k+1-m-g} )$, which completes  the inductive step for the value of $\rho(r_{k+1}^{\mathfrak{gl}(n+1)}(0,b)).$

\bigskip

Finally, we consider the value of $\rho(r_k^{\mathfrak{gl}(n+1)}(0,0))$.
\begin{eqnarray*}
\rho(r_{k+1}^{\mathfrak{gl}(n+1)}(0,0)) &= & \rho(r_k^{\mathfrak{gl}(n+1)}(0,0)E_{00}+\sum_{i>0}r_k^{\mathfrak{gl}(n+1)}(0,i)E_{i0}) \\
& = & \rho(r_k^{\mathfrak{gl}(n+1)}(0,0))\rho(E_{00})+\sum_{i>0}\rho(r_k^{\mathfrak{gl}(n+1)}(0,i))\rho(E_{i0}) \\
& = &    \Biggl( \left( \sum_{g=0}^{k-1} \binom{k}{g} R_1^g R_2^{k-1-g} \right)\left( t_0 \partial_0 \otimes 1 \right) + R_1^k\\
& - &  \sum_{m=2}^k \left( \sum_{g=0}^{k-m} \binom{k}{g} R_1^g R_2^{k-m-g} \right)\left( \sum_{i,j>0} \partial_i t_j \otimes r_{m-1}^{\mathfrak{gl}(n)}(i,j) \right) \Biggr)(t_0 \partial_0 \otimes 1 + R_1) \\
& + & \sum_{i>0} \Biggl( \sum_{m=1}^k \left( f_m(0,i) \sum_{g=0}^{k-m} \binom{k}{g} R_1^g R_2^{k-m-g} \right) \Biggr)(t_i \partial_0 \otimes 1 - \sum_{j>0} \frac{t_j}{t_0} \otimes E_{ij})\\
& = &  \left( \sum_{g=0}^{k-1} \binom{k}{g} R_1^g R_2^{k-1-g} \right)\left( t_0 \partial_0 t_0 \partial_0 \otimes 1 \right)  + R_1^k(t_0\partial_0 \otimes 1)\\
& - & \sum_{m=2}^k \left( \sum_{g=0}^{k-m} \binom{k}{g} R_1^g R_2^{k-m-g} \right)\left( \sum_{i,j>0} \partial_i t_j \otimes r_{m-1}^{\mathfrak{gl}(n)}(i,j) \right)(t_0 \partial_0 \otimes 1) \\
& + &  \left( \sum_{g=0}^{k-1} \binom{k}{g} R_1^{g+1} R_2^{k-1-g} \right)\left( t_0 \partial_0 \otimes 1 \right)  + R_1^{k+1}\\
& - & \sum_{m=2}^k \left( \sum_{g=0}^{k-m} \binom{k}{g} R_1^{g+1} R_2^{k-m-g} \right)\left( \sum_{i,j>0} \partial_i t_j \otimes r_{m-1}^{\mathfrak{gl}(n)}(i,j) \right) \\
& + & \sum_{i>0} \Biggl( \sum_{m=1}^k \left( f_m(0,i) \sum_{g=0}^{k-m} \binom{k}{g} R_1^g R_2^{k-m-g} \right) \Biggr)(t_i \partial_0 \otimes 1) \\
& - & \sum_{i>0} \Biggl( \sum_{m=1}^k \left( f_m(0,i) \sum_{g=0}^{k-m} \binom{k}{g} R_1^g R_2^{k-m-g} \right) \Biggr)\left(\sum_{j>0} \frac{t_j}{t_0} \otimes E_{ij}\right). \\
\end{eqnarray*}

We write the last expression as $X_1 + X_2 - X_3 + X_4 + X_5 - X_6 + X_7 - X_8$.

Since $X_7 = \sum_{m=1}^k \left( \sum_{g=0}^{k-m} \binom{k}{g} R_1^g R_2^{k-m-g} \right)\left( \sum_{i,j>0} \partial_j t_i \otimes r_{m-1}^{\mathfrak{gl}(n)}(j,i) \right)(t_0 \partial_0 \otimes 1)$,  $X_7-X_3 = \left( \sum_{g=0}^{k-1} \binom{k}{g} R_1^g R_2^{k-1-g} \right)\left(\sum_{i>0} \partial_i t_i \otimes 1 \right)\left( t_0 \partial_0 \otimes 1 \right) = \left( \sum_{g=0}^{k-1} \binom{k}{g} R_1^g R_2^{k-1-g} \right)\left(R_2 - (t_0 \partial_0 \otimes 1) \right)\left( t_0 \partial_0 \otimes 1 \right)$. Then 
$X_7-X_3+X_1+X_2 = \left( \sum_{g=0}^{k} \binom{k}{g} R_1^g R_2^{k-g} \right)(t_0 \partial_0 \otimes 1)$. But since  $X_4= \\ \left(\sum_{g=1}^{k} \binom{k}{g-1} R_1^{g} R_2^{k-g}\right)(t_0 \partial_0 \otimes 1)$, $X_7-X_3+X_1+X_2+X_4 =  \left( \sum_{g=0}^{k} \binom{k+1}{g} R_1^g R_2^{k-g} \right)(t_0 \partial_0 \otimes 1)$.

We next have $X_8 = \sum_{m=1}^k \Biggl( \sum_{i,j>0} f_m(0,i) \left(\frac{t_j}{t_0} \otimes E_{ij}\right)\Biggr)\left(\sum_{g=0}^{k-m} \binom{k}{g} R_1^gR_2^{k-m-g}\right)$. For any positive integer $m$, $\sum_{i,j>0} f_m(0,i) \left(\frac{t_j}{t_0} \otimes E_{ij}\right) = \sum_{i,j>0} \left(\sum_{\ell>0} t_0 \partial_{\ell} \otimes r_{m-1}^{\mathfrak{gl}(n)}(\ell,i)\right) \left(\frac{t_j}{t_0} \otimes E_{ij}\right) = \sum_{j,\ell>0} \partial_{\ell}t_j \otimes \left(\sum_{i>0} r_{m-1}^{\mathfrak{gl}(n)}(\ell,i)E_{ij}\right) = \sum_{j,\ell>0} \partial_{\ell}t_j \otimes r_m^{\mathfrak{gl}(n)}(\ell,j)$. Then we have \\ $X_8 =   \sum_{m=2}^{k+1} \left(\sum_{g=0}^{k+1-m} \binom{k}{g} R_1^gR_2^{k+1-m-g}\right) \left(\sum_{i,j>0} \partial_it_j \otimes r_{m-1}^{\mathfrak{gl}(n)}(i,j)\right)$. Also, \\ $X_6= \sum_{m=2}^k \left( \sum_{g=1}^{k+1-m} \binom{k}{g-1} R_1^{g} R_2^{k+1-m-g} \right)\left( \sum_{i,j>0} \partial_i t_j \otimes r_{m-1}^{\mathfrak{gl}(n)}(i,j) \right)$, and therefore\\ $X_6+X_8= \sum_{m=2}^{k+1} \left(\sum_{g=0}^{k+1-m} \binom{k+1}{g} R_1^gR_2^{k+1-m-g}\right) \left(\sum_{i,j>0} \partial_it_j \otimes r_{m-1}^{\mathfrak{gl}(n)}(i,j)\right)$.

Combining the above,  $\rho(r_{k+1}^{\mathfrak{gl}(n+1)}(0,0)) = (X_7-X_3+X_1+X_2+X_4)+X_5-(X_6+X_8) = \left( \sum_{g=0}^{k} \binom{k+1}{g} R_1^g R_2^{k-g} \right)(t_0 \partial_0 \otimes 1) + R_1^{k+1} - \sum_{m=2}^{k+1} \left(\sum_{g=0}^{k+1-m} \binom{k+1}{g} R_1^gR_2^{k+1-m-g}\right) \left(\sum_{i,j>0} \partial_it_j \otimes r_{m-1}^{\mathfrak{gl}(n)}(i,j)\right).$ This completes the induction step for $\rho(r_{k+1}^{\mathfrak{gl}(n+1)}(0,0))$ and hence the proof of the theorem. \end{proof}

\end{document}